\documentclass[11pt]{article}
\usepackage{amssymb, amsmath, amsthm}

\usepackage{amsfonts,mathtools,mathrsfs,amssymb,amsthm,color}
\usepackage{natbib}
\usepackage{euscript}
\usepackage[colorlinks]{hyperref}
\usepackage{graphicx}
\usepackage{bm}
\usepackage{microtype}
\usepackage[dvipsnames]{xcolor}

\usepackage{custom}

\allowdisplaybreaks{}
\hypersetup{citecolor=cyan, colorlinks=true, linkcolor=blue!75!black}

\newtheorem{theorem}{Theorem}
\newtheorem{lemma}[theorem]{Lemma}
\newtheorem{corollary}[theorem]{Corollary}

\newtheorem{remark}[theorem]{Remark}

\newcommand{\N}{\mathcal{N}}

\renewcommand{\part}[2]{\frac{\partial #1}{\partial #2}}

\definecolor{purple}{rgb}{0.54, 0.17, 0.89}

\newcommand{\yongxin}[1]{{\textcolor{blue}{#1}}}

\title{Improved analysis for a proximal algorithm for sampling}
\author{Yongxin Chen\footnote{School of Aerospace Engineering, Georgia Institute of Technology. Email: \texttt{yongchen@gatech.edu}}
\and Sinho Chewi\footnote{
Department of Mathematics, 
Massachusetts Institute of Technology.
Email: \texttt{schewi@mit.edu}
} 
\and Adil Salim\footnote{
Microsoft Research.
Email: \texttt{salim@berkeley.edu}
} 
\and Andre Wibisono\footnote{
Department of Computer Science, Yale University.
Email: \texttt{andre.wibisono@yale.edu}
}
}
\date{\today}

\begin{document}

\maketitle

\begin{abstract}
    We study the proximal sampler of~\citet{LST21} and obtain new convergence guarantees under weaker assumptions than strong log-concavity: namely, our results hold for (1) weakly log-concave targets, and (2) targets satisfying isoperimetric assumptions which allow for non-log-concavity. We demonstrate our results by obtaining new state-of-the-art sampling guarantees for several classes of target distributions. We also strengthen the connection between the proximal sampler and the proximal method in optimization by interpreting the proximal sampler as an entropically regularized Wasserstein proximal method, and the proximal point method as the limit of the proximal sampler with vanishing noise.
\end{abstract}

 \tableofcontents


\section{Introduction}

The problem of sampling from a target density $\pi^X\propto\exp(-f)$ on $\R^d$ has seen a resurgence of interest due to its staple role in scientific computing~\citep{robertcasella2004mcmc}, as well as its surprising and deep connections with the field of optimization. Indeed, the standard Langevin algorithm can be viewed as a gradient flow of the Kullback-Leibler (KL) divergence on the space of probability measures equipped with the geometry of optimal transport, a perspective which has led to new analyses~\citep{durmusmajewski2019lmcconvex, salimrichtarik2020proximallangevin, ahnchewi2021mirrorlangevin} and algorithms~\citep{pereyra2016proximal, zhangetal2020mirror, dingli2021coordinatelmc, maetal2021nesterovmcmc} inspired by the theory of convex optimization.

Among the algorithms in the optimization toolkit, we focus on \emph{proximal} methods. Classically, proximal methods are used to minimize composite objectives of the form $f + g$, where $g$ is smooth and convex and $f$ is non-smooth but simple enough to allow for evaluation of the proximal map $\prox_f: y\mapsto \argmin_{x\in\R^d}\{f(x) + \frac{1}{2\eta} \, \norm{x-y}^2\}$. However, the setting of our investigation is more closely related to the minimization of a non-composite objective $f$, for which the proximal method is known as the \emph{proximal point algorithm}~\citep{martinet1970breve,rockafellar1976monotone}.

As a natural first step towards developing a proximal point algorithm for sampling, one can combine the proximal map with the standard Langevin algorithm, leading to the \emph{proximal Langevin algorithm}. This algorithm was introduced in~\citet{pereyra2016proximal} and analyzed in the papers~\citet{bernton2018proximal,wibisono2019proximal,salimrichtarik2020proximallangevin}. 
Although these results are encouraging, the analogy between optimization methods and Langevin-based algorithms is imperfect because the discretization of the latter leads to asymptotic \emph{bias}, a feature which is typically not present in optimization (see~\citet{wibisono2018samplingoptimization} for a thorough discussion).

Remarkably, a new proximal algorithm for sampling was proposed recently in~\citet{LST21} which overcomes this issue via a novel Gibbs sampling approach.
Briefly, the \emph{proximal sampler} is a sampling algorithm which assumes access to samples from an oracle distribution, known as the \emph{restricted Gaussian oracle} (RGO); the RGO is a sampling analogue of the proximal map from optimization. Under this assumption, as well as the additional assumption that the target $\pi^X$ is strongly log-concave,~\citet{LST21} proved\footnote{There is an error in the conference version of the paper which is fixed in the arXiv version~\citep{LST21arxiv}.} that the proximal sampler converges exponentially fast to $\pi^X$ in total variation distance. In their paper, the proximal sampler was used as a \emph{reduction framework} to improve the condition number dependence of other sampling algorithms. Indeed, the RGO is a better conditioned distribution than the target distribution, so that implementing the RGO is easier than solving the original sampling task. In turn, the reduction framework allowed them to establish improved complexity results for a variety of structured log-concave sampling problems. We review the proximal sampler and its implementability in Section~\ref{scn:proximal_sampler}.

\paragraph{Our contributions.} Prior to our work, the convergence of the proximal sampler was only known in the case when $\pi^X \propto \exp(-f)$ is strongly log-concave. In this paper, we greatly expand the classes of targets to which the proximal sampler is applicable by providing new convergence guarantees.

First, we consider the case when $f$ is weakly convex. We show that after $k$ iterations, the proximal sampler outputs a distribution whose KL divergence to the target is $O(1/k)$. Our proof is analogous to, and is inspired by, the corresponding guarantee for minimizing a weakly convex function (in particular, the $O(1/k)$ rate matches the optimization result).

Next, we assume that $\pi^X$ satisfies a \emph{functional inequality}, e.g.\ a Poincar\'e inequality or a log-Sobolev inequality. Such functional inequalities have been employed in the sampling literature as tractable settings for non-log-concave sampling; see~\citet{VW19, chewietal2021lmcpoincare}. For these distributions, we show that the proximal sampler converges to the target in R\'enyi divergence (or any other weaker metric, such as KL divergence) with a rate that matches the known convergence rates for the continuous-time Langevin diffusion under the same assumptions.

In each of these settings, if we additionally assume that $\nabla f$ is Lipschitz, then the RGO is implementable, as it becomes a smooth strongly log-concave distribution. Hence, we obtain new sampling guarantees for gradient Lipschitz potentials when the target is weakly log-concave or satisfies a functional inequality. In all cases, our results are \emph{stronger} than known results in the literature.

Finally, we clarify the connection between the proximal sampler and the proximal point algorithm in optimization in the following ways: (1) We show that convergence proofs for the proximal sampler can be translated to yield convergence proofs for the proximal point algorithm. As a consequence, we obtain a new convergence guarantee for the proximal point method under a gradient domination condition with optimal rate, which is (to the best of our knowledge) a new result. (2) We show that the RGO can be interpreted as a proximal mapping on Wasserstein space, and that the proximal sampler can be interpreted as an entropically regularized Wasserstein proximal method (i.e.\ JKO scheme). The latter perspective allows us to recover the proximal point algorithm as a certain limit of the proximal sampler as the ``noise level'' (corresponding to the entropic regularization) tends to zero.

\paragraph{Organization.} The rest of the paper is organized as follows. We begin with background on distances between probability measures in Section~\ref{scn:background} and on the proximal sampler in Section~\ref{scn:proximal_sampler}.

We give our main results in Section~\ref{scn:results}. In particular, we state our new convergence guarantees for the proximal sampler in Section~\ref{scn:new_conv_results}, and we give applications of our results in Section~\ref{scn:applications}. We then describe the connections between the proximal sampler and the proximal point method in Section~\ref{scn:prox_sampler_and_ppm}. All proofs are given in Section~\ref{scn:pfs}.

Finally, we conclude and list open directions in Section~\ref{scn:conclusion}.

\section{Background and notation}\label{scn:background}

Throughout the paper, we abuse notation by identifying a probability measure with its density w.r.t.\ Lebesgue measure.
For a probability measure $\rho \ll \pi$, we define the \emph{KL divergence}, the \emph{chi-squared divergence}, and the \emph{R\'enyi divergence} of order $q\ge 1$ respectively via
\begin{align*}
    H_\pi(\rho)
    &:= \int \rho \log \frac{\rho}{\pi}\,, \qquad \chi^2_\pi(\rho)
    := \int \frac{\rho^2}{\pi} - 1\,, \qquad R_{q,\pi}(\rho)
    := \frac{1}{q-1}\log \int \frac{\rho^q}{\pi^{q-1}}\,,
\end{align*}
with $R_{1,\pi} = H_\pi$. We recall that for $1 \le q \le q' < \infty$, we have the monotonicity property $R_{q,\pi} \le R_{q',\pi}$, and that $R_{2,\pi} = \log(1 + \chi^2_\pi)$.

We also define the $2$-Wasserstein distance between $\rho$ and $\pi$ to be
\begin{align*}
    W_2^2(\rho,\pi)
    &:= \inf_{\gamma \in \eu C(\rho,\pi)} \int \norm{x-y}^2 \, d \gamma(x,y)\,,
\end{align*}
where $\eu C(\rho,\pi)$ is the set of \emph{couplings} of $\rho$ and $\pi$, i.e.\ joint distributions on $\R^d\times\R^d$ whose marginals are $\rho$ and $\pi$.
We refer readers to~\citet{villani2003topics} for an introduction to optimal transport, and to~\citet{ambrosio2008gradient} for a detailed treatment of Wasserstein calculus.

\section{The proximal sampler}\label{scn:proximal_sampler}

Our goal is to sample from a target probability distribution $\pi^X$ on $\R^d$ with density $\pi^X \propto \exp(-f)$ and finite second moment, where $f \colon \R^d \to \R$ is the \emph{potential}.

Following~\citet{LST21}, we define the joint target distribution $\pi$ on $\R^d \times \R^d$ with density
$$\pi(x,y) \propto \exp\Bigl(-f(x) - \frac{1}{2\eta} \, \norm{x-y}^2\Bigr)\,,$$ 
where $\eta > 0$ is the \emph{step size} of the algorithm.

Observe that the $X$-marginal of $\pi$ is equal to the original target distribution $\pi^X$, whereas the conditional distribution of $Y$ given $X$ is Gaussian:
$\pi^{Y \mid X}(\cdot \mid x) = \N(x, \eta I)$.
Therefore, the $Y$-marginal is the convolution of $\pi^X$ with a Gaussian, $\pi^Y = \pi^X \ast \N(0, \eta I)$. The perspective that we adopt in our proofs is that $\pi^Y$ is obtained by evolving $\pi^X$ along the heat flow for time $\eta$.

The conditional distribution of $X$ given $Y$ is the ``regularized'' distribution
$$\pi^{X \mid Y}(x \mid y) \propto_x \exp\Bigl(-f(x) - \frac{1}{2\eta} \, \norm{x-y}^2\Bigr)\,.$$

The \emph{restricted Gaussian oracle} (RGO) is defined as an oracle that, given $y \in \R^d$, outputs a random variable distributed according to $\pi^{X \mid Y}(\cdot\mid y)$.
We also write $\pi^{X \mid Y}(\cdot\mid y) = \pi^{X \mid Y=y}$. 

\paragraph{Proximal Sampler:}
The proximal sampler is initialized at a point $x_0\in\R^d$ and performs Gibbs sampling on the joint target $\pi$.
That is, the proximal sampler iterates the following two steps:
\begin{enumerate}
    \item From $x_k$, sample $y_k \mid x_k \sim \pi^{Y \mid X}(\cdot \mid x_k) \,=\, \N(x_k, \eta I)$.
    \item From $y_k$, sample $x_{k+1} \mid y_k \sim \pi^{X \mid Y}(\cdot \mid y_k)$.
\end{enumerate}
The first step consists in sampling a Gaussian random variable centered at $x_k$, and is therefore easy to implement. The second step calls the RGO at the point $y_k$.

As is well-known from the theory of Gibbs sampling, the iterates ${(x_k,y_k)}_{k\in\mathbb N}$ form a reversible Markov chain with stationary distribution $\pi$. That is, the proximal sampler is an \emph{unbiased} sampling algorithm, unlike algorithms based on discretizations of stochastic processes such as the unadjusted Langevin algorithm. This is because the proximal sampler is an idealized algorithm in which we assume \emph{exact} access to the RGO\@.
For our applications, we implement the RGO via rejection sampling; see Section~\ref{scn:applications} for details and Section~\ref{Sec:Gaussian} for an explicit example in the Gaussian case.

\section{Results}\label{scn:results}

\subsection{New convergence results for the proximal sampler}\label{scn:new_conv_results}

In this section, we describe our new convergence results for the proximal sampler under various assumptions, beginning with the strongly log-concave and weakly log-concave cases, and then proceeding to targets satisfying functional inequalities which allow for non-log-concavity.

\subsubsection{Strong log-concavity}

We start by recalling the $W_2$ contraction result from~\citet{LST21arxiv} for the proximal sampler under strong log-concavity.

\begin{theorem}[{\citet[Lemma 2]{LST21arxiv}}]\label{Thm:SLC}
Assume that $\pi^X \propto \exp(-f)$ is $\alpha$-strongly log-concave (\textit{i.e.}, $f$ is $\alpha$-strongly convex), where $\alpha \geq 0$.
For any $\eta > 0$ and for any two initial distributions $\rho^X_0$, $\bar \rho^X_0$, after $k$ iterations of the proximal sampler with step size $\eta$, the respective distributions $\rho^X_k$, $\bar \rho^X_k$ satisfy
\begin{align}
    W_2(\rho^X_k, \bar \rho^X_k) \le \frac{W_2(\rho^X_0, \bar\rho^X_0)}{{(1 + \alpha \eta)}^{k}}\,.
\end{align}
\end{theorem}

Although this result was stated in~\citet{LST21arxiv} as a convergence result rather than a contraction, the latter is implicit in the proof. From the proof of~\citet{LST21arxiv}, one can also read off a convergence guarantee in KL divergence, although this will be a corollary of our result in Section~\ref{scn:lsi_result}.

We revisit Theorem~\ref{Thm:SLC} in Section~\ref{scn:pf_slc} and provide a proof which more closely resembles a classical convergence proof of the proximal point algorithm. We use Wasserstein subdifferential calculus.

We note that this is the sampling analogue of the classical fact that the proximal map for an $\alpha$-strongly convex function with step size $\eta$ is a $\frac{1}{1+\alpha\eta}$-contraction. In Appendix~\ref{scn:contractivity_proximal_map}, we give a new proof of this fact by translating the proof of~\citet{LST21arxiv} into optimization.

\subsubsection{Log-concavity}

The preceding result does not yield convergence when $\alpha = 0$. We provide a new convergence guarantee for the weakly convex case which closely resembles the optimization guarantee for minimizing weakly convex functions~\citep[see][Theorem 3.3]{bubeck2015convex}.

\begin{theorem}\label{Thm:LC}
    Assume that $\pi^X \propto \exp(-f)$ is log-concave (\textit{i.e.}, $f$ is convex).
    For the $k$-th iterate $\rho_k^X$ of the proximal sampler,
    \begin{align*}
        H_{\pi^X}(\rho_k^X)
        \le \frac{W_2^2(\rho_0^X, \pi^X)}{k\eta}\,.
    \end{align*}
\end{theorem}
\begin{proof}
    Section~\ref{Sec:LC}.
\end{proof}

\subsubsection{Log-Sobolev inequality}\label{scn:lsi_result}

Recall that a probability distribution $\pi$ satisfies the log-Sobolev inequality (LSI) with constant $\alpha > 0$ ($\alpha$-LSI) if for any probability distribution $\rho$, the following inequality holds:
\begin{equation}\label{eq:LSI}
    H_\pi (\rho) \le \frac{\alpha}{2} \,J_\pi (\rho)\,.
\end{equation}
Here $J_\pi(\rho)$ is the Fisher information of $\rho$ w.r.t.\ $\pi$; see Section~\ref{Sec:LSI}.
Recall that strong log-concavity implies LSI, and that LSI is equivalent to the gradient domination 
condition for relative entropy $H_\pi$~\citep{ottovillani2000lsi}; see also Section~\ref{scn:lsipl}.

\begin{theorem}\label{Thm:LSI}
Assume that $\pi^X \propto \exp(-f)$ satisfies $\alpha$-LSI.
For any $\eta > 0$ and any initial distribution $\rho_0^X$, the $k$-th iterate $\rho_k^X$ of the proximal sampler with step size $\eta$ satisfies
\begin{align}
    H_{\pi^X}(\rho_k^X) \le \frac{H_{\pi^X}(\rho_0^X)}{{(1 + \alpha \eta)}^{2k}}\,.
\end{align}
Furthermore, for all $q \ge 1$:
\begin{align}
    R_{q,\pi^X}(\rho_k^X) \le \frac{R_{q,\pi^X}(\rho_0^X)}{{(1 + \alpha \eta)}^{2k/q}}\,.
\end{align}
\end{theorem}
\begin{proof}
Section~\ref{Sec:LSI}.
\end{proof}

\subsubsection{Poincar\'e inequality}

Recall that a probability distribution $\pi$ satisfies the Poincar\'e inequality (PI) with constant $\alpha > 0$ ($\alpha$-PI) if for any smooth bounded function $\psi : \R^d\to\R$, the following inequality holds:
\begin{align}\label{eq:pi}
    \var_\pi(\psi)
    &\le \frac{1}{\alpha} \E_\pi[\norm{\nabla \psi}^2]\,.
\end{align}
Recall also that LSI implies PI.

\begin{theorem}\label{Thm:Poincare}
Assume $\pi^X \propto \exp(-f)$ satisfies $\alpha$-PI.
For any $\eta > 0$ and any initial distribution $\rho_0^X$, the $k$-th iterate $\rho_k^X$ of the proximal sampler with step size $\eta$ satisfies
\begin{align}
    \chi_{\pi^X}^2(\rho^X_k) \le \frac{\chi_{\pi^X}^2(\rho^X_0)}{{(1 + \alpha \eta)}^{2k}}\,.
\end{align}
Furthermore, for all $q \ge 2$,
\begin{align}
    R_{q,\pi^X}(\rho^X_k) \le \begin{cases} R_{q,\pi^X}(\rho_0^X) - \frac{2k \log(1+\alpha\eta)}{q}\,, &\text{if}~k\le \frac{q}{2\log(1+\alpha\eta)} \, (R_{q,\pi^X}(\rho_0^X) - 1)\,, \\
    1/{(1+\alpha\eta)}^{2(k-k_0)/q}\,, &\text{if}~k \ge k_0 := \lceil \frac{q}{2\log(1+\alpha\eta)} \, (R_{q,\pi^X}(\rho_0^X) - 1)\rceil\,.
    \end{cases}
\end{align}
\end{theorem}
\begin{proof}
    Section~\ref{Sec:Poincare}.
\end{proof}

\subsubsection{Lata\l{}a--Oleszkiewicz inequality}

We next consider a family of functional inequalities which interpolate between PI and LSI.
A probability distribution $\pi$ satisfies the Lata\l{}a--Oleszkiewicz inequality (LOI) of order $r \in [1,2]$ and constant $\alpha > 0$ ($(r,\alpha)$-LOI) if for any smooth bounded function $\psi : \R^d\to\R_+$, the following inequality holds:
\begin{align*}
    \sup_{p\in (1,2)} \frac{\var_{p,\pi}(\psi)}{{(2-p)}^{2 \, (1-1/r)}}
    &:= \sup_{p\in (1,2)} \frac{\E_\pi[\psi^2] - {\E_\pi[\psi^p]}^{2/p}}{{(2-p)}^{2 \, (1-1/r)}}
    \le \frac{1}{\alpha} \E_\pi[\norm{\nabla \psi}^2]\,.
\end{align*}
This inequality was introduced in~\citet{latalaoleszkiewicz2000loineq}, and sampling guarantees for the Langevin algorithm under LOI were given in~\citet{chewietal2021lmcpoincare}. The LOI for $r = 1$ is equivalent to PI and the LOI for $r = 2$ is equivalent to LSI, up to absolute constants. Generally speaking, $(r,\alpha)$-LOI captures targets $\pi\propto\exp(-f)$ such that the tails of $f$ grow as $\norm\cdot^r$.

\begin{theorem}\label{thm:loi}
Assume $\pi^X \propto \exp(-f)$ satisfies $(r,\alpha)$-LOI with $r \in [1,2)$.
For any $\eta > 0$, $q\ge 2$, and any initial distribution $\rho_0^X$, the $k$-th iterate $\rho_k^X$ of the proximal sampler with step size $\eta$ satisfies
\begin{align}
    R_{q,\pi^X}(\rho^X_k) \le \begin{cases} \Bigl({R_{q,\pi^X}(\rho_0^X)}^{2/r-1} - \frac{(2/r-1)\,k \log(1+\alpha\eta)}{68q}\Bigr)^{r/(2-r)}\,, &\text{if}~k\le c_0\,, \\
    1/{(1+\alpha\eta)}^{(k-\lceil c_0\rceil)/(68q)}\,, &\text{if}~k \ge \lceil c_0 \rceil\,.
    \end{cases}
\end{align}
where
\begin{align*}
    c_0 := \frac{68q}{(2/r-1)\log(1+\alpha\eta)} \, \bigl({R_{q,\pi^X}(\rho_0^X)}^{2/r-1} - 1\bigr)\,.
\end{align*}
(For $r = 2$, we can instead use Theorem~\ref{Thm:LSI}.)
\end{theorem}
\begin{proof}
    Section~\ref{scn:loi_pf}.
\end{proof}

To interpret the result, suppose that $R_{q,\pi^X}(\rho_0^X) = O(d)$ at initialization and that $\eta \ll 1/\alpha$. Then, the theorem states that after an initial waiting period of $\lceil c_0\rceil = O(d^{2/r-1}/\eta)$ iterations, in which the R\'enyi divergence decays to $O(1)$, the R\'enyi divergence decays exponentially thereafter. This interpolates between a waiting time of $O(d/\eta)$ under PI ($r = 1$; Theorem~\ref{Thm:Poincare}) and a waiting time of $O((\log d)/\eta)$ under LSI ($r=2$; Theorem~\ref{Thm:LSI}).

\subsection{Applications of the convergence results}\label{scn:applications}

We start with a corollary of Theorem~\ref{Thm:LC}. Suppose that $f$ is $\beta$-smooth, i.e.\ $\nabla f$ is $\beta$-Lipschitz. Then, provided $\frac{1}{\eta} \ge \beta$, the RGO $\pi^{X\mid Y}$ is strongly-log-concave, with condition number $(1+\beta\eta)/(1-\beta\eta)$. We can implement the RGO via rejection sampling.

\paragraph{Rejection Sampling:} Given a target distribution $\tilde \pi \propto \exp(-\tilde f)$, where $\tilde f$ is $\tilde \alpha$-strongly convex, perform the following steps.
\begin{enumerate}
    \item Compute the minimizer $x^\star$ of $\tilde f$.
    \item Repeat until acceptance: draw a random variable $Z \sim \mc N(x^\star, \tilde \alpha^{-1} I)$ and accept it with probability $\exp(-\tilde f(Z) + \tilde f(x^\star) + \frac{\tilde\alpha}{2} \, \norm{Z - x^\star}^2)$.
\end{enumerate}
The resulting sample is distributed according to $\tilde \pi$, and one can show that the expected number of iterations of the algorithm is bounded by $\tilde \kappa^{d/2}$ with $\tilde \kappa := \tilde \beta/\tilde \alpha$ and $\tilde\beta$ is the smoothness of $\tilde f$; see e.g.~\citet[Theorem 7]{chewietal2021logconcave1d}.

We apply this to $\tilde f$ given by $\tilde f(x) = f(x) + \frac{1}{2\eta} \, \norm{x-y}^2$. The algorithm above requires exact minimization of $\tilde f$, which we assume for simplicity (since it is well-known how to efficiently minimize a strongly convex and smooth function). With the choice $\eta \asymp \frac{1}{\beta d}$, the expected number of iterations is $O(1)$. Combining this implementation of the RGO with Theorem~\ref{Thm:LC}, we obtain:

\begin{corollary}
Suppose $\pi^X \propto \exp(-f)$ where $f$ is convex and $\beta$-smooth.
    Take $\eta \asymp \frac{1}{\beta d}$ and implement the RGO with rejection sampling as described above.
    Then, the proximal sampler outputs $\rho_k^X$ with $H_{\pi^X}(\rho_k^X) \le \varepsilon$ and the expected number of calls to an oracle for $f$ is $O(\beta d\, W_2^2(\rho_0^X, \pi^X)/\varepsilon)$.
\end{corollary}

More precisely, our algorithm requires access to an oracle of $f$ which can evaluate $f$ and compute the proximity operator for $f$.

We now compare this rate with others in the literature. Let $\mf m_2$ denote the second moment of $\pi^X$.
For example, $\mf m_2 = O(d)$ for a product measure, and $\mf m_2 = O(d^2)$ when $f(x) = \sqrt{1+\norm x^2}$. It is reasonable to assume that the Poincar\'e constant $\alpha$ of $\pi^X$ is $\Omega(d/\mf m_2)$ and that $W_2^2(\rho_0^X, \pi^X) = O(\mf m_2)$. With these simplifications, our complexity is $O(\beta d \mf m_2/\varepsilon)$; averaged LMC achieves $\widetilde O(\beta d \mf m_2/\varepsilon^2)$~\citep{durmusmajewski2019lmcconvex}; MALA achieves $\widetilde O(\beta^{3/2} d^{1/2} \mf m_2^{3/2}/\varepsilon^{3/4})$ albeit in ${\rm TV}^2$~\citep{dwivedi2019log, chenetal2020hmc}; and LMC achieves $\widetilde O(\beta^2 \mf m_2^2/\varepsilon)$ in the stronger R\'enyi metric~\citep{chewietal2021lmcpoincare}. Since all these complexity results also hold in terms of the squared total variation distance, our result has arguably the state-of-the-art complexity for this setting (at least, if dimension dependence is the primary consideration).

Similarly, implementing the RGO with rejection sampling in Theorem~\ref{thm:loi} yields:

\begin{corollary}
    Suppose $\pi^X \propto \exp(-f)$ where $f$ is $\beta$-smooth and $\pi^X$ satisfies $(r, \alpha)$-LOI.
    Take $\eta \asymp \frac{1}{\beta d}$ and implement the RGO with rejection sampling as described above.
    Then, the proximal sampler outputs $\rho_k^X$ with $R_{q,\pi^X}(\rho_k^X) \le \varepsilon $ and the expected number of calls to an oracle for $f$ is $\widetilde O(\frac{\beta dq}{\alpha}\, (R_{q,\pi^X}(\rho_0^X)^{2/r-1} \vee \log(1/\varepsilon)))$.
\end{corollary}

Even for the special case of a Poincar\'e inequality and smoothness, the first sampling guarantee under these assumptions is quite recent~\citep{chewietal2021lmcpoincare}. Let us write $\hat\kappa := \beta/\alpha$ for the ``condition number'' and assume $R_{q,\pi^X}(\rho_0^X) = O(d)$~\citep[see e.g.][Appendix A]{chewietal2021lmcpoincare}. Then, our complexity is $\widetilde O(\hat\kappa d q \, (d^{2/r-1} \vee \log(1/\varepsilon)))$, whereas~\citet[Theorem 7]{chewietal2021lmcpoincare} gives a complexity bound for LMC of order $\widetilde O(\hat\kappa^2 d^{4/r-1} q^3/\varepsilon)$. We note that our result is the \emph{first} high-accuracy guarantee for this setting (i.e.\ the complexity depends polylogarithmically on $\varepsilon$). Moreover, even in the low-accuracy regime $\varepsilon \asymp 1$, our complexity of $\widetilde O(\hat \kappa d^{2/r} q)$ is always better (e.g.\ in the Poincar\'e case $r=1$, our rate is $\widetilde O(\hat\kappa d^2 q )$ whereas~\citet{chewietal2021lmcpoincare} yields $\widetilde O(\hat\kappa^2 d^3 q^3)$), although we note that~\citet{chewietal2021lmcpoincare} handles the more general weakly smooth case.

Surprisingly, the same strategy of rejection sampling also applies to non-smooth potentials. In \citet{LiaChe21}, it was shown that when the above rejection sampling is applied to $\tilde f(x) = f(x) + \frac{1}{2\eta} \, \norm{x-y}^2$ with $f(x)$ being a convex and $M$-Lipschitz function, if $\eta\le 1/(16M^2d)$, the expected number of iterations of the algorithm is bounded above by $2$. Moreover, the result is insensitive to the inexactness of the minimizer of $\tilde f$ \citep{LiaChe21}. Combining it with Theorem \ref{Thm:LC} and Theorem~\ref{Thm:Poincare} we establish:
\begin{corollary}\label{cor:convex_lip_result}
Suppose $\pi^X \propto \exp(-f)$ where $f$ is convex and $M$-Lipschitz. Take $\eta \asymp \frac{1}{M^2 d}$ and implement the RGO with rejection sampling as described above.
\begin{enumerate}
    \item Applying Theorem~\ref{Thm:LC}, we deduce that the proximal sampler outputs $\rho_k^X$ with $H_{\pi^X}(\rho_k^X) \le \varepsilon$ and the expected number of calls to an oracle for $f$ is $O(M^2 d\, W_2^2(\rho_0^X, \pi^X)/\varepsilon)$.
    \item Applying Theorem~\ref{Thm:Poincare} (using the fact that log-concave measures satisfy $\alpha$-PI for some $\alpha > 0$), we deduce that the proximal sampler outputs $\rho_k^X$ with $R_{q,\pi^X}(\rho_k^X) \le \varepsilon$ and the expected number of calls to an oracle for $f$ is $O(\frac{M^2 dq}{\alpha} \, (R_{q,\pi^X}(\rho_0^X) \vee \log(1/\varepsilon)))$.
\end{enumerate}
\end{corollary}

We make the same simplifications as above to compare the rates. Our complexity (from the second part of Corollary~\ref{cor:convex_lip_result} is $O(M^2 \mf m_2 \, (d \vee \log(1/\varepsilon)))$, whereas~\citet{durmusmajewski2019lmcconvex} achieves $O(M^2 \mf m_2/\varepsilon^2)$ in KL divergence and~\citet{LiaChe21} achieves $\widetilde O(M^2 d\mf m_2/\varepsilon^{1/2})$ in squared total variation distance.
In particular, when $\mf m_2 = O(d)$, our result is the state-of-the-art.

\subsection{On the relation between the proximal sampler and the proximal point algorithm}\label{scn:prox_sampler_and_ppm}

The proximal sampler is motivated by the proximal point method in optimization.
Recall that in optimization, the proximal point method for minimizing $f$ is the iteration of the proximal mapping
\begin{align}\label{eq:prox}
    \prox_{\eta f}(y)
    &:= \argmin_{x\in\R^d}{\Bigl\{ f(x) + \frac{1}{2\eta} \, \norm{x-y}^2\Bigr\}}\,
\end{align}
with some step size $\eta > 0$. Formally, using the correspondence $f \leftrightarrow \exp(-f)$ between optimization and sampling, the RGO can be viewed as the sampling analogue of the proximal mapping.

In this section, we establish a more precise correspondence between the proximal sampler algorithm (for sampling from $\exp(-f)$) and the proximal point method (for minimizing $f$).

\subsubsection{Convergence under LSI/PL}
\label{scn:lsipl}

We recall that LSI for $\pi \propto \exp(-f)$ is equivalent to the statement that the relative entropy $H_\pi$ satisfies the gradient domination condition (or the Polyak-\L{}ojasiewicz (PL) inequality) in the Wasserstein metric~\citep{ottovillani2000lsi}.
Thus, in the optimization setting, the analogous assumption to LSI is that $f$ satisfies PL.

We recall $f$ satisfies the PL inequality with constant $\alpha > 0$ ($\alpha$-PL) if for all $x$, 
$$\|\nabla f(x)\|^2 \ge 2\alpha \,(f(x) - f^\star),$$ 
where $f^\ast = \inf f$. 
The PL inequality allows for mild non-convexity of $f$, yet still implies exponential convergence of gradient flow or proximal point method for minimizing $f$; see for example~\citep{KNS16}.

In light of our convergence guarantee for the proximal sampler under LSI in Theorem~\ref{Thm:LSI}, it is natural to ask whether there is an analogous result for the proximal point method under PL.
We answer this affirmatively via the following theorem.
We note that a less careful proof of the argument gives the suboptimal contraction factor $\frac{1}{1+\alpha\eta}$; to the best of our knowledge, we are not aware of another reference which obtains the optimal contraction factor under PL~\citep{attouch2009convergence}.\footnote{The optimality of our bound can be obtained by considering $f(x) = \frac{\alpha}{2}\norm{x}^2$.}

\begin{theorem}\label{Thm:ProxLSI}
    Suppose that $f: \R^d \to (-\infty,+\infty]$ is differentiable and satisfies $\alpha$-PL and let $x' \in \prox_{\eta f}(x)$. Also, write $f^\star = \inf f$.
    Then, it holds that
    \begin{align*}
        f(x') - f^\star
        &\le \frac{1}{{(1+\alpha \eta)}^2} \, \{f(x) - f^\star\}\,.
    \end{align*}
\end{theorem}
\begin{proof}
Section~\ref{Sec:ProofProxLSI}.
\end{proof}

\subsubsection{RGO as a proximal operator on Wasserstein space}

Consider $y \in \R^d$. Noting that $\pi^{X\mid Y=y}(dx) \propto_{x} \exp(-\frac{1}{2\eta}\norm{x-y}^2) \pi^X(dx)$ and using~\citet[Remark 9.4.2]{ambrosio2008gradient} we have $$H_{\pi^{X}}(\rho^X) = H_{\pi^{X\mid Y=y}}(\rho^X) - \int \frac{1}{2\eta}\,\norm{x-y}^2 \,d\rho^X(x) + C(y)\,,$$ where $C(y)$ is a constant depending only on $y$. Using $\argmin H_{\pi^{X\mid Y=y}}(\cdot) = \pi^{X\mid Y=y}$, the RGO can be expressed as
\begin{equation}\label{eq:rgo_as_prox}
\begin{aligned}
    \pi^{X\mid Y=y}
    &= \argmin_{\rho^X \in \mc P_2(\R^d)} \Bigl\{H_{\pi^{X}}(\rho^X) + \frac{1}{2\eta} \int \|x-y\|^2\, d\rho^X(x)\Bigr\} \\
    &= \argmin_{\rho^X \in \mc P_2(\R^d)}\Bigl\{ H_{\pi^{X}}(\rho^X) + \frac{1}{2\eta} \,W_2^2(\rho^X,\delta_y)\Bigr\}\,.
    \end{aligned}
\end{equation}
Thus, by replacing the Euclidean distance by the Wasserstein distance, $\pi^{X\mid Y=y} = \prox_{\eta \, H_{\pi^{X}}}(\delta_y)$. We use this fact in Section~\ref{scn:pf_slc} to provide a new proof of the contraction of the proximal sampler under strong log-concavity (Theorem~\ref{Thm:SLC}). The proximal operator over the Wasserstein space is also known as the JKO scheme~\citep{JorKinOtt98}, which we describe further in the next section.

\subsubsection{Proximal sampler as entropy-regularized JKO scheme}
\label{Sec:ProxLimitIntro}

The Wasserstein gradient flow models the steepest descent dynamics of a functional $F$ over the space of probability distributions with respect to the $2$-Wasserstein distance $W_2$. One strategy to approximate the Wasserstein gradient flow in discrete time is the JKO scheme \citep{JorKinOtt98}, which follows the iterations
    \begin{equation}\label{eq:JKO}
        \mu_{k+1} = \argmin_{\mu\in\mc P_2(\R^d)} {\Bigl\{F(\mu) + \frac{1}{2\eta}\, W_2^2(\mu_k, \mu)\Bigr\}}\,,
    \end{equation}
where $\eta>0$ is the step size. Note that this is a Wasserstein analogue of the proximal point method. A variant of the JKO scheme with an extra entropic regularization term was developed in \citet{peyre2015entropic} to improve the computational efficiency. In this entropy-regularized Wasserstein gradient flow algorithm, one instead follows the update
    \begin{equation}\label{eq:JKO_entropy}
        \mu_{k+1} = \argmin_{\mu\in\mc P_2(\R^d)}{\Bigl\{F(\mu) + \frac{1}{2\eta} \,W_{2,\epsilon}^2(\mu_k, \mu)\Bigr\}}\,,
    \end{equation}
where $W_{2,\epsilon}$ is the entropy-regularized $2$-Wasserstein distance
defined as
	\begin{equation}\label{eq:entropyW}
		W_{2,\epsilon}^2(\mu,\nu) := \min_{\gamma \in \eu C(\mu,\nu)}{\Bigl\{\int \|x-y\|^2 \, d\gamma(x,y) + \epsilon H(\gamma)\Bigr\}}\,,
	\end{equation}
where $H(\gamma) = \int \gamma\log \gamma$ denotes the negative entropy.

We show that proximal sampler can be viewed as an entropy-regularized JKO scheme in the following result. 
\begin{theorem}\label{thm:JKO}
Let $\rho_k^X, \rho_k^Y, \rho_{k+1}^X$ be the distributions of $x_k, y_k, x_{k+1}$, respectively, in one iteration of the proximal sampler algorithm.
Then, they follow the entropy-regularized JKO scheme
    \begin{equation}\label{eq:forwardJKO}
        \rho_k^Y = \argmin_{\mu\in\mc P_2(\R^d)} \frac{1}{2\eta} \,W_{2,2\eta}^2(\rho_k^X, \mu)\,,
    \end{equation}
and 
    \begin{equation}\label{eq:backwardJKO}
        \rho_{k+1}^X = \argmin_{\mu\in\mc P_2(\R^d)}{\Bigl\{ \int f \, d\mu + \frac{1}{2\eta} \,W_{2,2\eta}^2 (\rho_k^Y,\mu)\Bigr\}}\,.
    \end{equation}
\end{theorem}
\begin{proof}
Section~\ref{Sec:EntReg}.
\end{proof}

\subsubsection{Proximal point method as the limit of the proximal sampler}
\label{Sec:ProxLimitIntro2}

The interpretation of the proximal sampler algorithm above provides some insights on its connections to optimization. 
We can define a more general family of proximal sampler algorithm with a different level of entropy regularization. 
The forward step is
	\begin{equation}\label{eq:forwardJKOgeneral}
		\rho_k^Y = \argmin_{\mu\in\mc P_2(\R^d)} \frac{1}{2\eta} \,W_{2,2\eta \epsilon}^2(\rho_k^X, \mu)\,,
	\end{equation}
and the backward step reads
	\begin{equation}\label{eq:backwardJKOgeneral}
		\rho_{k+1}^X = \argmin_{\mu\in\mc P_2(\R^d)}{\Bigl\{ \int f \, d\mu + \frac{1}{2\eta} \,W_{2,2\eta \epsilon}^2 (\rho_k^Y,\mu)\Bigr\}}\,.
	\end{equation}

\begin{theorem}\label{Thm:Limit}
As $\epsilon\searrow 0$, \eqref{eq:forwardJKOgeneral}--\eqref{eq:backwardJKOgeneral} reduces to the proximal point algorithm in optimization. 
\end{theorem}
\begin{proof}
Section~\ref{Sec:ProxLimit}.
\end{proof}

Indeed, when $\epsilon = 0$, with $\rho_k^X = \delta_{x_k}$, we have
	\[
		\rho_k^Y = \rho_k^X = \delta_{x_k}
	\]
and furthermore, $\rho_{k+1}^X = \delta_{x_{k+1}}$ with 
	\[
		x_{k+1} = \argmin_{x\in\R^d}{\Bigl\{ f(x)+\frac{1}{2\eta}\,\|x-x_k\|^2\Bigr\}}\,.
	\]
This is exactly the proximal point method. In fact, even if $\rho_k^X$ is not a Dirac distribution, \eqref{eq:forwardJKOgeneral}--\eqref{eq:backwardJKOgeneral} with $\epsilon=0$ can be viewed as a parallel implementation of the proximal point method with many different initial points.
See Section~\ref{Sec:ProxLimit} for more discussion.

\subsection{Example: Gaussian case}
\label{Sec:Gaussian}

Suppose that the target distribution is a Gaussian $\N(0, \Sigma)$, i.e., $f(x) = \frac{1}{2}\,\langle x,\Sigma^{-1}x \rangle$.
In this case we can compute the iterations of the proximal sampler explicitly.

If we initialize the proximal sampler at 
\begin{align*}
    \rho^X_0 = \N(m_0, \Sigma_0)\,,
\end{align*}
then some calculations show that
\begin{align*}
    \rho^Y_k
    &= \N(m_k, \Sigma_k +\eta I)\,, \\
    \rho^X_{k+1}
    &= \N(m_{k+1}, \Sigma_{k+1})\,,
\end{align*}
where\footnote{
We can also notice that $m_{k+1} = \prox_{\eta f}(m_k),$ i.e., the means of the distributions follow the proximal point algorithm for $f$. Moreover, $m_k \to 0 = \argmin f$ which is the mean of the target distribution.}
\begin{align*}
    m_{k+1}
    &:= \Sigma\, {(\Sigma + \eta I)}^{-1}\, m_k\,, \\
    \Sigma_{k+1}
    &:= \Sigma \,{(\Sigma + \eta I)}^{-1} \,(\Sigma_k + \eta I)\,{(\Sigma + \eta I)}^{-1}\, \Sigma + \eta \,\Sigma \,{(\Sigma + \eta I)}^{-1}\,.
\end{align*}
Specializing to the case where $\Sigma = I$, $\eta = 1$, and we initialize at $\N(0, \sigma_0^2 I)$, we obtain
\begin{align}
\label{eq:gauss}
    \abs{\sigma_k^2 - 1}
    &= \frac{\abs{\sigma_0^2-1}}{4^k}\,.
\end{align}
In particular, this shows that the contraction factor $\frac{1}{{(1+\alpha\eta)}^2}$ in Theorem~\ref{Thm:LSI} is sharp.

\section{Conclusion and open directions}\label{scn:conclusion}

In this paper, we have studied in detail the proximal sampler of~\citet{LST21}. In particular, we have given new convergence proofs under weaker assumptions than what were previously considered, allowing for a much wider class of distributions beyond log-concavity.
In some cases, our proofs are inspired by convex optimization; in others, they show a remarkable parallel with the continuous-time theory of the Langevin diffusion under isoperimetry. Additionally, we have drawn more precise links between the proximal sampler and the proximal point method in optimization.

We conclude by listing a few directions for future study.
\begin{enumerate}
    \item Is there an extension of the theory we have developed to the problem of sampling from composite potentials $\pi^X \propto \exp(-f-g)$?
    \item Is there an accelerated version of the proximal sampler?
\end{enumerate}

\paragraph{Acknowledgments.} We would like to thank Ruoqi Shen and Kevin Tian for helpful conversations. YC was supported in part by grants NSF CAREER ECCS-1942523 and NSF CCF-2008513. SC was supported by the Department of Defense (DoD) through the National Defense Science \& Engineering Graduate Fellowship (NDSEG) Program. AS was supported by a Berkeley--Simons Research Fellowship. This work was done while the authors were visiting the Simons Institute for the Theory of Computing.

\appendix

\section{Proofs for the proximal sampler}\label{scn:pfs}

\subsection{Techniques}

At a high level, our proofs proceed by considering the change in KL divergence or R\'enyi divergence when we apply the following two operations to the law $\rho_k^X$ of the iterate and the target $\pi^X$: (1) we simultaneously evolve the two measures along the heat flow for time $\eta$, and then (2) we apply the RGO to the resulting measures.

For the first step, we formulate a remarkably general lemma in Section~\ref{scn:simultaneous_flow} which shows that the computation of the time derivative of \emph{any} $\phi$-divergence along the simultaneous heat flow is similar (in a precise sense) to the analogous computation when studying the continuous-time Langevin diffusion. It is this property that allows us to apply functional inequalities which are usually used for the Langevin diffusion, such as the Poincar\'e and log-Sobolev inequalities, in order to study the convergence of the proximal sampler.

In the second step, we are applying the same operation (of sampling from the RGO) to each measure, so the data-processing inequality implies that the KL divergence or R\'enyi divergence can only decrease. Combined with the previous step, it is sufficient to prove a convergence guarantee for the proximal sampler; however, the rate turns out to be suboptimal. In order to recover the optimal rate, we introduce an argument based on the \emph{Doob $h$-transform} (described in Section~\ref{scn:doob}) to obtain contraction in the second step as well, using the backward version of our general lemma (see Section~\ref{scn:simultaneous_backward_flow}).
We summarize our technique in Section~\ref{scn:general_strategy}.

\subsubsection{Lemma on the simultaneous heat flow}\label{scn:simultaneous_flow}

Let $\Phi_\pi$ be a $\phi$-divergence for some convex function $\phi$, i.e.\
\begin{align*}
    \Phi_\pi(\rho) := \E_\pi\bigl[\phi\bigl(\frac{\rho}{\pi}\bigr)\bigr]\,.
\end{align*}
We assume that $\phi$ is regular enough to justify the interchange of differentiation and integration and to perform integration by parts; this is satisfied for all of our applications.

We will use the following result in each forward step of the proximal sampler.
This is a generalization of~\citet[Lemma~16]{VW19}.

\begin{lemma}\label{Lem:SimFlow}
Let ${(\mu^X_t)}_{t\ge 0}$ be the law of the continuous-time Langevin diffusion with target distribution $\pi^X$, and define the \emph{dissipation functional} $D_{\pi^X}$ via the time derivative of $\Phi_{\pi^X}$ along the diffusion:
\begin{align*}
    D_{\pi^X}(\mu^X_t)
    &:= -\partial_t \Phi_{\pi^X}(\mu^X_t)
    = \E_{\mu_t^X} \Bigl\langle \nabla\bigl( \phi' \circ \frac{\mu_t^X}{\pi^X}\bigr), \nabla \log \frac{\mu_t^X}{\pi^X} \Bigr\rangle \,.
\end{align*}
If ${(\rho^X Q_t)}_{t\ge 0}$ and ${(\pi^X Q_t)}_{t\ge 0}$ evolve according to the simultaneous heat flow,
\begin{align*}
    \partial_t \rho^X Q_t
    &= \frac{1}{2} \, \Delta (\rho^X Q_t)\,, \qquad
    \partial_t \pi^X Q_t
    = \frac{1}{2} \, \Delta (\pi^X Q_t)\,,
\end{align*}
then
\begin{align*}
    \partial_t \Phi_{\pi^X Q_t}(\rho^X Q_t)
    &= -\frac{1}{2} \, D_{\pi^X Q_t}(\rho^X Q_t)\,.
\end{align*}
\end{lemma}
\begin{proof}
    On one hand, we know that ${(\mu_t^X)}_{t\ge 0}$ satisfies the Fokker-Planck equation
    \begin{align*}
        \partial_t \mu_t^X
        &= \divergence\bigl(\mu_t^X \nabla \log \frac{\mu_t^X}{\pi^X}\bigr)
    \end{align*}
    so that
    \begin{align*}
        \partial_t \Phi_{\pi^X}(\mu_t^X)
        &= \int \phi'\bigl( \frac{\mu_t^X}{\pi^X}\bigr) \, \partial_t\mu_t^X
        = \int \phi'\bigl( \frac{\mu_t^X}{\pi^X}\bigr) \divergence\bigl(\mu_t^X \nabla \log \frac{\mu_t^X}{\pi^X}\bigr) \\
        &= -\int \Bigl\langle \nabla \bigl[ \phi'\bigl( \frac{\mu_t^X}{\pi^X}\bigr) \bigr], \nabla \log \frac{\mu_t^X}{\pi^X}\Bigr\rangle \, \mu_t^X\,.
    \end{align*}
    
    On the other hand, writing $\rho_t^X := \rho^X Q_t$ and $\pi_t^X := \pi^X Q_t$ for brevity, along the simultaneous heat flow we compute
    \begin{align*}
        2\,\partial_t \Phi_{\pi_t^X}(\rho_t^X)
        &= 2\int \phi'\bigl( \frac{\rho_t^X}{\pi_t^X}\bigr) \, \Bigl( \partial_t \rho_t^X - \frac{\rho_t^X}{\pi_t^X} \, \partial_t \pi_t^X \Bigr) + 2 \int \phi\bigl(\frac{\rho_t^X}{\pi_t^X}\bigr) \, \partial_t \pi_t^X \\
        &= \int \phi'\bigl( \frac{\rho_t^X}{\pi_t^X}\bigr) \, \Bigl( \divergence(\rho_t^X \nabla \log \rho_t^X) - \frac{\rho_t^X}{\pi_t^X} \divergence(\pi_t^X \nabla \log \pi_t^X) \Bigr) \\
        &\qquad\qquad\qquad{} + \int \phi\bigl(\frac{\rho_t^X}{\pi_t^X}\bigr) \divergence(\pi_t^X \nabla \log \pi_t^X) \\
        &= -\int \Bigl\langle \nabla\bigl[\phi'\bigl( \frac{\rho_t^X}{\pi_t^X}\bigr)\bigr], \nabla \log \rho_t^X \Bigr\rangle \, \rho_t^X + \int\Bigl\langle \nabla\bigl[\phi'\bigl( \frac{\rho_t^X}{\pi_t^X}\bigr) \, \frac{\rho_t^X}{\pi_t^X}\bigr], \nabla \log \pi_t^X\Bigr\rangle \, \pi_t^X \\
        &\qquad\qquad\qquad{} - \int \Bigl\langle \nabla\bigl[\phi\bigl(\frac{\rho_t^X}{\pi_t^X}\bigr)\bigr], \nabla \log \pi_t^X \Bigr\rangle \, \pi_t^X \\
        &= -\int \Bigl\langle \nabla\bigl[\phi'\bigl( \frac{\rho_t^X}{\pi_t^X}\bigr)\bigr], \nabla \log \frac{\rho_t^X}{\pi_t^X} \Bigr\rangle \, \rho_t^X + \int\Bigl\langle \nabla \frac{\rho_t^X}{\pi_t^X}, \nabla \log \pi_t^X\Bigr\rangle \, \phi'\bigl( \frac{\rho_t^X}{\pi_t^X}\bigr) \, \pi_t^X \\
        &\qquad\qquad\qquad{} - \int \Bigl\langle \nabla\frac{\rho_t^X}{\pi_t^X}, \nabla \log \pi_t^X \Bigr\rangle \,\phi'\bigl(\frac{\rho_t^X}{\pi_t^X}\bigr)\, \pi_t^X \\
        &= -D_{\pi_t^X}(\rho_t^X)\,. 
    \end{align*}
\end{proof}

\begin{remark}
    A similar statement holds if we replace the $\phi$-divergence $\Phi_\pi$ with any function $\psi \circ \Phi_\pi$ of the $\phi$-divergence.
    This allows us to cover the R\'enyi divergence introduced in Section~\ref{scn:background}.
\end{remark}

\subsubsection{Doob's $h$-transform}\label{scn:doob}

Doob's $h$-transform is a useful method to analyze the properties of a diffusion process conditioned on its value at some terminal time point. Consider a general diffusion process modeled by the stochastic differential equation (SDE)
    \begin{equation}\label{eq:SDEforward}
        dZ_t = b(t, Z_t)\, dt + \sigma(t, Z_t)\, dW_t\,, \qquad Z_0 \sim \mu_0\,,
    \end{equation}
where ${(W_t)}_{t\ge 0}$ denotes a standard Wiener process.
Assume that $b(t,z)$ and $\sigma(t,z)$ are piecewise continuous with respect to $t$ and Lipschitz continuous with respect to $z$ so that the above SDE~\eqref{eq:SDEforward} has a unique solution. The Doob $h$-transform characterizes the process conditional on its terminal value $Z_T$, summarized in the following lemma~\citep{SarSol19}.

\begin{lemma}\label{lem:doob_new}
Let ${(\hat Z_t)}_{0 \le t\le T}$ be the process~\eqref{eq:SDEforward} conditioned to satisfy $Z_T = z$. Then, the process satisfies the following SDE backwards in time: 
\begin{align*}
    d \hat Z_t
    &= [b(t, \hat Z_t) - \sigma(t,\hat Z_t)\, \sigma(t,\hat Z_t)^\T \,\nabla \log \mu_t(\hat Z_t)]\, dt + \sigma(t, \hat Z_t)\, dW_t\,,
\end{align*}
where $\mu_t$ is the marginal distribution of $Z_t$ in~\eqref{eq:SDEforward} and the SDE is started with $\hat Z_T = z$.

Equivalently, if we define the SDE 
\begin{equation}\label{eq:htransform}
    d \hat Z_t^-
    = [\yongxin{-}b(T-t, \hat Z_t^-) + \sigma(T-t,\hat Z_t^-)\, \sigma(T-t,\hat Z_t^-)^\T \,\nabla \log \mu_{T-t}(\hat Z_t^-)]\, dt + \sigma(T-t, \hat Z_t^-)\, dW_t\,,
\end{equation}
started at $\hat Z_0^- = z$, then at time $T$ the law of $\hat Z_T^-$ is the conditional distribution of $Z_0$ given $Z_T = z$.
\end{lemma}

\subsubsection{Lemma on the simultaneous backward heat flow}\label{scn:simultaneous_backward_flow}

We present the following backward version of Lemma~\ref{Lem:SimFlow}, which we use in each backward step of the proximal sampler.
We assume the same set up as in Lemma~\ref{Lem:SimFlow}:
Let $\Phi_\pi(\rho) = \E_\pi[\phi(\frac{\rho}{\pi})]$ be a $\phi$-divergence for some convex function $\phi$, i.e.\
\begin{align*}
    \Phi_\pi(\rho) := \E_\pi\bigl[\phi\bigl(\frac{\rho}{\pi}\bigr)\bigr]
\end{align*}
and let
\begin{align*}
    D_{\pi}(\rho)
    = \E_{\rho} \Bigl\langle \nabla\bigl( \phi' \circ \frac{\rho}{\pi}\bigr), \nabla \log \frac{\rho}{\pi} \Bigr\rangle \,
\end{align*}
so that $D_{\pi}$ is the dissipation of $\Phi_\pi$ along the Langevin dynamics with target $\pi$.

\begin{lemma}\label{lem:simflow2_new}
    Let $\pi^X$ be a probability distribution and let $\pi(x,y) = \pi^X(x) \, \N(y; x, \eta I)$ be a joint density for $(X,Y)$ with $Y$ obtained from $X$ by running the heat flow for time $\eta$. Let $\pi^{X\mid Y}$ be the conditional distribution of $X$ given $Y$ under $\pi$, and let $\pi^Y$ denote the marginal distribution of $Y$. Then, for each $t\in [0,\eta]$, there exists a channel $Q_t^-$ that maps probability measures to probability measures, with the following properties: (1) $Q_0^-$ is the identity channel; (2) $Q_\eta^-$ maps a probability measure $\rho^Y$ to the the measure $\rho^Y Q_\eta^-(x) = \int \pi^{X\mid Y}(x\mid y) \,\rho^Y(dy)$; (3) for every $t$, $\pi^Y Q_t^- = \pi \ast \N(0, (\eta - t) I)$; and (4) for every $\rho^Y$,
    \begin{align*}
        \partial_t \Phi_{\pi^Y Q_t^-}(\rho^Y Q_t^-)
        = - \frac{1}{2} \, D_{\pi^Y Q_t^-}(\rho^Y Q_t^-)\,.
    \end{align*}
\end{lemma}

The channel is obtained from the Doob $h$-transform.
\medskip{}

\begin{proof}
    Let $\pi_t := \pi^X \ast \N(0, tI)$. We define $Q_t^-$ as follows: given $\rho^Y$, we set $\rho^Y Q_t^-$ to be the law at time $t$ of the SDE
    \begin{align}\label{eq:diffusion_in_simul_backwards}
        d \hat Z_t^-
        &= \nabla \log \pi_{\eta-t}(\hat Z_t^-)\, dt + dW_t\,,
    \end{align}
    started at $\hat Z_0^- \sim \rho^Y$. According to Lemma~\ref{lem:doob_new} applied to the Brownian motion process (started at $\pi^X$), the channels ${(Q_t^-)}_{0\le t \le \eta}$ satisfy properties (1), (2), and (3). It remains to verify (4).
    In the proof, we write $\pi_t^- := \pi^Y Q_t^-$ and $\rho_t^- := \rho^Y Q_t^-$ for brevity.
    Note that $\pi_{\eta-t} = \pi_t^-$ by construction, and we have the Fokker-Planck equations:
    \begin{align*}
        \partial_t \pi_t^-
        &= -\divergence(\pi_t^- \nabla \log \pi_t^-) + \frac{1}{2} \, \Delta \pi_t^- = -\frac{1}{2} \, \Delta \pi_t^-\,, \\
        \partial_t \rho_t^-
        &= -\divergence(\rho_t^- \nabla \log \pi_t^-) + \frac{1}{2} \, \Delta \rho_t^-
        = \divergence\bigl(\rho_t^- \nabla \log \frac{\rho_t^-}{\pi_t^-}\bigr) - \frac{1}{2} \, \Delta \rho_t^-\,.
    \end{align*}
    Hence,
    \begin{align*}
        2\,\partial_t \Phi_{\pi_t^-}(\rho_t^-)
        &= 2\int \phi'\bigl( \frac{\rho_t^-}{\pi_t^-}\bigr) \, \Bigl( \partial_t \rho_t^- - \frac{\rho_t^-}{\pi_t^-} \, \partial_t \pi_t^- \Bigr) + 2 \int \phi\bigl(\frac{\rho_t^-}{\pi_t^-}\bigr) \, \partial_t \pi_t^- \\
        &= \int \phi'\bigl( \frac{\rho_t^-}{\pi_t^-}\bigr) \, \Bigl(
        2\divergence\bigl(\rho_t^- \nabla \log \frac{\rho_t^-}{\pi_t^-} \bigr)
        - \Delta \rho_t^- + \frac{\rho_t^-}{\pi_t^-}\, \Delta \pi_t^- \Bigr) - \int \phi\bigl(\frac{\rho_t^-}{\pi_t^-}\bigr) \,\Delta \pi_t^- \\
        &= 2 \int \phi'\bigl( \frac{\rho_t^-}{\pi_t^-}\bigr)
        \divergence\bigl(\rho_t^- \nabla \log \frac{\rho_t^-}{\pi_t^-} \bigr)   \\
        &\qquad{} - {\underbrace{\int \phi'\bigl( \frac{\rho_t^-}{\pi_t^-}\bigr) \, \Bigl(
        \Delta \rho_t^- - \frac{\rho_t^-}{\pi_t^-}\, \Delta \pi_t^- \Bigr) + \int \phi\bigl(\frac{\rho_t^-}{\pi_t^-}\bigr) \,\Delta  \pi_t^-}_{= -D_{\pi_t^-}(\rho_t^-) ~ \text{ by Lemma~\ref{Lem:SimFlow}}}} \\
        &= -2 \int \Bigl\langle \nabla\bigl[\phi'\bigl( \frac{\rho_t^-}{\pi_t^-}\bigr)\bigr], \nabla \log \frac{\rho_t^-}{\pi_t^-} \Bigr\rangle \, \rho_t^- + D_{\pi_t^-}(\rho_t^-) \\
        &= -2D_{\pi_t^-}(\rho_t^-) + D_{\pi_t^-}(\rho_t^-)
        = -D_{\pi_t^-}(\rho_t^-)\,. 
    \end{align*}
\end{proof}

\subsubsection{General strategy of the proofs}\label{scn:general_strategy}

Suppose that we want to understand the change in the $\phi$-divergence $\Phi_{\pi^X}(\rho^X_1)$ after one iteration of the proximal sampler, compared to the $\phi$ divergence $\Phi_{\pi^X}(\rho^X_0)$ at initialization. We split the analysis into two steps.
\begin{enumerate}
    \item {\bf Forward step:}
    In the first step, we draw $y_0 \mid x_0 \sim \pi^{Y \mid X=x_0} = \N(x_0, \eta I)$. 

    This creates a joint distribution $\rho_0(x,y) = \rho_0^X(x) \,  \N(y;x,\eta I)$
    with the correct conditionals: $\rho_0^{Y\mid X} = \pi^{Y \mid X}$.
    Therefore, the $\phi$-divergence of the joint distribution is equal to the initial $\phi$-divergence of the $X$-marginal:
    $\Phi_{\pi}(\rho_0) = \Phi_{\pi^X}(\rho_0^X)$.

    Consider the $Y$-marginal $y_0 \sim \rho_0^Y$.
    Observe that $\rho_0^Y = \rho_0^X \ast \N(0, \eta I)$ is the output $\rho_0^Y = \tilde \rho_\eta$ of the heat flow $\partial_t \tilde \rho_t  = \frac{1}{2} \Delta \tilde \rho_t$ at time $t = \eta$ starting from $\tilde \rho_0 = \rho_0^X$.
    We denote this by $\rho_0^Y = \rho_0^X Q_\eta$, where $(Q_t)_{t \ge 0}$ denotes the heat semigroup.
    
    Similarly, we can write the $Y$-marginal of the target as $\pi^Y = \pi^X \ast \N(0, \eta I) = \pi^X Q_\eta$.
    
    In particular, ${(\rho_0^X Q_t)}_{t\ge 0}$ and ${(\pi^X Q_t)}_{t\ge 0}$ evolve following the simultaneous heat flow.
    
    By Lemma~\ref{Lem:SimFlow}, along the simultaneous heat flow, 
    \begin{align*}
       \partial_t \Phi_{\pi^X Q_t}(\rho_0^X Q_t) = -\frac{1}{2}\, D_{\pi^X Q_t}(\rho_0^X Q_t)
   \end{align*}
   where $D_\cdot(\cdot)$ denotes the dissipation functional for the $\phi$-divergence along the Langevin dynamics. Hence, a lower bound on $D_{\pi^X Q_t}(\rho_0^X Q_t)$ leads to an upper bound on
   \begin{align*}
       \Phi_{\pi^Y}(\rho_0^Y) - \Phi_{\pi^X}(\rho_0^X)
       = \Phi_{\pi^X Q_\eta}(\rho_0^X Q_\eta) - \Phi_{\pi^X}(\rho_0^X)\,.
   \end{align*}
   \item \textbf{Backward step:} In the second step, we draw $x_1 \mid y_0 \sim \pi^{X \mid Y=y_0}$.

    This time, we consider the backward heat flow and apply Lemma~\ref{lem:simflow2_new}, which yields the Doob channels ${(Q_t^-)}_{0\le t\le\eta}$ with $\rho_1^X = \rho_0^Y Q_\eta^-$ and $\pi^X = \pi^Y Q_\eta^-$.
    Lemma~\ref{lem:simflow2_new} implies that
    \begin{align*}
       \partial_t \Phi_{\pi^Y Q_t^-}(\rho_0^Y Q_t^-) = -\frac{1}{2}\, D_{\pi^Y Q_t^-}(\rho_0^Y Q_t^-)\,.
   \end{align*}
   Observe that this is almost symmetric with the forward step!
   In particular, a lower bound on $D_{\pi^Y Q_t^-}(\rho_0^Y Q_t^-)$ leads to an upper bound on
   \begin{align*}
       \Phi_{\pi^X}(\rho_1^X) - \Phi_{\pi^Y}(\rho_0^Y)
       = \Phi_{\pi^Y Q_\eta^-}(\rho_0^Y Q_\eta^-) - \Phi_{\pi^Y}(\rho_0^Y)\,.
   \end{align*}
\end{enumerate}
Combining the two steps allows to understand each iteration of the proximal sampler.

\subsection{Convergence under strong log-concavity}\label{scn:pf_slc}

Suppose that $A$ is a set-valued mapping on $\R^d$ which is strongly monotone, in the sense that
\begin{align*}
    \langle A(x) - A(y), x-y \rangle \ge \alpha \, \norm{x-y}^2 \qquad\text{for all}~x,y\in\R^d\,.
\end{align*}
Suppose that $x' \in x - \eta A(x')$ and $y' \in y - \eta A(y')$. Then, by expanding out the square, one can easily show that $\norm{x'-y'}^2 \le \frac{1}{{(1+\alpha\eta)}^2} \, \norm{x-y}^2$. In particular, by applying this to the subdifferential $A = \partial f$, where $f$ is $\alpha$-strongly convex, one immediately obtains the fact that the proximal point algorithm is a $\frac{1}{1+\alpha\eta}$-contraction. In this section, we translate this proof to the sampling setting.

Recall from~\eqref{eq:rgo_as_prox} that $\pi^{X\mid Y=y} = \prox_{\eta F}(\delta_y)$, where $F = H_{\pi^{X}}$ is $\alpha$-geodesically strongly strongly convex~\citep[Equation 10.1.8]{ambrosio2008gradient}. 
Then, from the first-order optimality conditions on Wasserstein space~\citep[see][Lemma 10.1.2]{ambrosio2008gradient}, we have
\begin{equation}\label{eq:first_order_opt_prox}
0 \in \partial F(\pi^{X\mid Y=y}) + \frac{1}{\eta}\,({\id} - y)\,, \qquad \pi^{X\mid Y=y}\text{-a.s.},
\end{equation}
where $\partial F$ denotes the Wasserstein subdifferential of $F$.

\begin{proof}[Proof of Theorem~\ref{Thm:SLC}]
First, let $y, \bar y \in \R^d$.
Then, from~\eqref{eq:first_order_opt_prox}:
\begin{align}
    \id &\in y - \eta \, \partial F(\pi^{X\mid Y=y})\,, & \pi^{X|Y=y}\text{-a.s.} \label{eq:first_order_opt_1} \\
    \id &\in \bar y - \eta \, \partial F(\pi^{X\mid Y=\bar y})\,, & \pi^{X|Y=\bar y}\text{-a.s.} \label{eq:first_order_opt_2}
\end{align}
Let $T$ be the optimal transport map from $\pi^{X\mid Y=y}$ to $\pi^{X\mid Y=\bar y}$. We can rewrite~\eqref{eq:first_order_opt_2} as 
\begin{equation}\label{eq:first_order_opt_3}
    T \in \bar y - \eta \, \partial F(\pi^{X\mid Y=\bar y})\circ T\,, \qquad \pi^{X\mid Y=y}\text{-a.s.}
\end{equation}

We now abuse notation and write $\partial F(\pi^{X\mid Y=y})$ for an element of the subdifferential. Then, using~\eqref{eq:first_order_opt_1} and~\eqref{eq:first_order_opt_3}, $\pi^{X\mid Y=y}$-a.s.,
\begin{align*}
    \norm{T- {\id}}^2 = \norm{\bar y - y}^2
    &-2\eta\, \langle \partial F(\pi^{X\mid Y=\bar y}) \circ T - \partial F(\pi^{X\mid Y=y}), T-{\id} \rangle \\
    &- \eta^2 \, \|\partial F(\pi^{X\mid Y=\bar y}) \circ T - \partial F(\pi^{X\mid Y=y})\|^2\,.
\end{align*}
Integrating with respect to $\pi^{X\mid Y=y}$, and using the geodesic strong convexity of $F$~\citep[Equation 10.1.8]{ambrosio2008gradient}, 
\begin{align*}
    W^2(\pi^{X\mid Y=y},\pi^{X\mid Y=\bar y})
    &\leq \norm{y - \bar y}^2
    -2 \alpha \eta \, W^2(\pi^{X\mid Y=y},\pi^{X\mid Y=\bar y})   -\alpha^2 \eta^2 \, W^2(\pi^{X\mid Y=y},\pi^{X\mid Y=\bar y})\,.
\end{align*}
Therefore, 
$$W^2(\pi^{X\mid Y=y},\pi^{X\mid Y=\bar y}) \leq \frac{1}{{(1+\alpha \eta)}^2}\,\|y - \bar y\|^2.$$

The rest of the argument is concluded as in~\citet[Lemma 2]{LST21arxiv}. We provide the details here for completeness.
First, along the proximal sampler, we have $W_2(\rho_0^Y, \bar \rho_0^Y) \le W_2(\rho_0^X, \bar \rho_0^X)$ because the heat flow is a Wasserstein contraction (see Section~\ref{scn:general_strategy} for the notation). Next, let $\gamma$ denote an optimal coupling of $\rho_0^Y$ and $\bar\rho_0^Y$, and for all $y,y\in\R^d$ let $\gamma_{y,\bar y}$ denote an optimal coupling of $\pi^{X\mid Y=y}$ and $\pi^{X\mid Y=\bar y}$. We check that the measure $\hat \gamma(dx, d\bar x) := \gamma(dy, d\bar y) \, \gamma_{y,\bar y}(dx, d\bar x)$ is a valid coupling of $\rho_1^X$ and $\bar\rho_1^X$. To check that, for instance, the first marginal of $\hat \gamma$ is $\rho_1^X$, we take a bounded measurable function $\psi : \R^d\to\R$ and calculate
\begin{align*}
    \int \psi(x) \, \hat\gamma(dx, d\bar x)
    &= \iint \psi(x) \, \gamma(dy, d\bar y) \, \gamma_{y,\bar y}(dx, d\bar x)
    = \iint \psi(x) \, \gamma(dy, d\bar y) \, \pi^{X\mid Y=y}(dx) \\
    &= \iint \psi(x) \, \rho_0^Y(dy) \, \pi^{X\mid Y=y}(dx)
    = \int \psi(x) \, \rho_1^X(dx)\,,
\end{align*}
and similarly the second marginal of $\hat \gamma$ is $\bar\rho_1^X$. Therefore,
\begin{align*}
    W_2^2(\rho_1^X,\bar\rho_1^X)
    &\le \int \norm{x-\bar x}^2 \, \hat \gamma(dx,d\bar x)
    = \iint \norm{x-\bar x}^2 \, \gamma(dy, d\bar y) \, \gamma_{y,\bar y}(dx, d\bar x) \\
    &= \int W_2^2(\pi^{X\mid Y=y}, \pi^{X\mid Y=\bar y}) \, \gamma(dy, d\bar y) \\
    &\le \frac{1}{{(1+\alpha \eta)}^2} \int \norm{y-\bar y}^2 \, \gamma(dy, d\bar y)
    = \frac{1}{{(1+\alpha \eta)}^2} \, W_2^2(\rho_0^Y,\bar\rho_0^Y)\,,
\end{align*}
which completes the proof.
\end{proof}

\subsection{Convergence under log-concavity}
\label{Sec:LC}

For a probability distribution $\rho$ with smooth relative density $\frac{\rho}{\pi}$, the {\em Fisher information} of $\rho$ with respect to $\pi$ is
\begin{align}\label{eq:fisher_info}
    J_\pi(\rho) := \int \rho \,\Bigl\| \nabla \log \frac{\rho}{\pi} \Bigr\|^2
    = \E_\pi\Bigl[ \frac{\pi}{\rho} \,\Bigl\| \nabla \frac{\rho}{\pi} \Bigr\|^2 \Bigr]\,.
\end{align}
Recall that Fisher information is the dissipation of KL divergence along the Langevin dynamics.

\begin{proof}[Proof of Theorem~\ref{Thm:LC}]
    We follow the strategy and notation of Section~\ref{scn:general_strategy}.
    \begin{enumerate}
        \item \textbf{Forward step:} By log-concavity of $\pi^X Q_t$ (since log-concavity is preserved by convolution~\citep{saumard2014logconcavity}), the convexity of $H_{\pi^X Q_t}$ along Wasserstein geodesics~\citep[Theorem 9.4.11]{ambrosio2008gradient} yields the inequality
    \begin{align*}
        0
        &= H_{\pi^X Q_t}(\pi^X Q_t) \\
        &\ge H_{\pi^X Q_t}(\rho_0^X Q_t) + \E_{(X_t, Y_t) \sim \msf{OPT}(\rho_0^X Q_t,\pi^X Q_t)}\bigl\langle \nabla \log \frac{\rho_0^X Q_t}{\pi^X Q_t}(X_t), Y_t - X_t \bigr\rangle
    \end{align*}
    where $\msf{OPT}(\cdot,\cdot)$ is used to denote the optimal transport plan.
    Hence,
    \begin{align}\label{Eq:LCCalc}
        \underbrace{\E_{\rho_0^X Q_t}\bigl[\bigl\lVert \nabla \log \frac{\rho_0^X Q_t}{\pi^X Q_t} \bigr\rVert^2\bigr]}_{= J_{\pi^XQ_t}(\rho_0^XQ_t)} \, W_2^2(\rho_0^X Q_t,\pi^X Q_t)
        &\ge {H_{\pi^X Q_t}(\rho_0^X Q_t)}^2\,.
    \end{align}
    So, by Lemma~\ref{Lem:SimFlow} and~\eqref{Eq:LCCalc},
    \begin{align*}
        \partial_t H_{\pi^X Q_t}(\rho_0^X Q_t)
        &= -\frac{1}{2}\,J_{\pi^XQ_t}(\rho_0^XQ_t)
        \le -\frac{1}{2}\, \frac{{H_{\pi^X Q_t}(\rho_0^X Q_t)}^2}{W_2^2(\rho_0^X Q_t,\pi^X Q_t)}\,.
    \end{align*}
    Also, observe that $t\mapsto W_2^2(\rho_0^X Q_t,\pi^X Q_t)$ is decreasing because the heat flow is a $W_2$ contraction (which can be proven directly quite easily).
    Solving this differential inequality yields
    \begin{align*}
        \frac{1}{H_{\pi^Y}(\rho_0^Y)}
        =
        \frac{1}{H_{\pi^X Q_\eta}(\rho_0^X Q_\eta)}
        &\ge \frac{1}{H_{\pi^X}(\rho_0^X)} + \frac{\eta}{2W_2^2(\rho_0^X,\pi^X)}\,.
    \end{align*}
    \item \textbf{Backward step:} By Lemma~\ref{lem:simflow2_new} and~\eqref{Eq:LCCalc},
    \[
        \partial_t H_{\pi^Y Q_t^-}(\rho_0^Y Q_t^-) = -\frac{1}{2} \,J_{\pi^Y Q_t^-}(\rho_0^Y Q_t^-)
        \le -\frac{1}{2} \, \frac{{H_{\pi_Y Q_t^-}(\rho_0^Y Q_t^-)}^2}{W_2^2(\rho_0^Y Q_t^-, \pi^Y Q_t^-)}\,. 
    \]
By~\eqref{eq:htransform}, the channels ${(Q_t^-)}_{t\ge 0}$ can be modeled by the diffusion
    \[
        dZ_t = \nabla \log \pi_{\eta-t}(Z_t)\,dt + dW_t.
    \]
Since $\log \pi_{\eta-t}$ is concave, with a standard coupling argument, one can show that $t\mapsto W_2(\rho_0^Y Q_t^-, \pi^Y Q_t^-)$ is decreasing. Hence,
\begin{align*}
    W_2(\rho_0^Y Q_t^-, \pi^Y Q_t^-) \le W_2(\rho_0^Y Q_0^-, \pi^Y Q_0^-) = W_2(\rho_0^Y, \pi^Y) \le W_2(\rho_0^X, \pi^X)\,.
\end{align*}
Therefore, we deduce that
    \[
        \frac{1}{H_{\pi^X}(\rho_1^X)}=\frac{1}{H_{\pi^Y Q_\eta^-}(\rho_0^Y Q_\eta^-)}
        \ge \frac{1}{H_{\pi^Y}(\rho_0^Y)} + \frac{\eta}{2 W_2^2(\rho_0^X,\pi^X)}\,.
    \]

    Finally, we iterate this inequality and recall that $W_2^2(\rho_k^X,\pi^X) \le W_2^2(\rho_0^X,\pi^X)$ for all $k\in\mathbb N$ (see Theorem~\ref{Thm:SLC} for $\alpha = 0$). It quickly yields
    \begin{align*}
        \frac{1}{H_{\pi^X}(\rho_k^X)}
        &\ge \frac{1}{H_{\pi^X}(\rho_0^X)} + \frac{k\eta}{W_2^2(\rho_0^X, \pi^X)}
    \end{align*}
    or
    \begin{align*}
        H_{\pi^X}(\rho_k^X)
        &\le \frac{H_{\pi^X}(\rho_0^X)}{1 + k\eta \,H_{\pi^X}(\rho_0^X)/W_2^2(\rho_0^X, \pi^X)}
        \le \frac{W_2^2(\rho_0^X, \pi^X)}{k\eta}\,. 
    \end{align*}
    \end{enumerate}
\end{proof}

\subsection{Convergence under LSI}\label{Sec:LSI}

We recall the following definitions.
For a probability distribution $\rho$ with smooth relative density $\frac{\rho}{\pi}$,
the {\em R\'enyi information} of $\rho$ with respect to $\pi$ of order $q \ge 1$ is
\begin{align*}
    J_{q,\pi}(\rho) := q\, \frac{\E_\pi \bigl[ \bigl(\frac{\pi}{\rho}\bigr)^{q-2}\, \bigr\| \nabla \frac{\rho}{\pi} \bigr\|^2 \bigr]}{
    \E_\pi\bigl[\bigl(\frac{\pi}{\rho}\bigr)^{q}\bigr]}\,.
\end{align*}
Note that $J_{1,\pi}(\rho) = J_\pi(\rho)$, where $J_\pi$ is the Fisher information~\eqref{eq:fisher_info}.
Recall that by definition, $\pi$ satisfies $\alpha$-LSI if for all $\rho$,
$J_\pi(\rho) \ge 2 \alpha H_\pi(\rho)$.
One can show this also implies for all $q \ge 1$:
\begin{align}\label{Eq:RenyiLSI}
    J_{q,\pi}(\rho) \ge \frac{2\alpha}{q} \,R_{q,\pi}(\rho)\,,
\end{align}
see for example~\citet[Lemma~5]{VW19}.
Just as Fisher information is the dissipation of KL divergence along the Langevin dynamics,  R\'enyi information is the dissipation of R\'enyi divergence along the Langevin dynamics.

\begin{proof}[Proof of Theorem~\ref{Thm:LSI}]
We will prove the following one-step improvement lemma for R\'enyi divergence of order $q \ge 1$: 
For any initial distribution $\rho_0^X$, after one iteration of the proximal sampler with step size $\eta > 0$, the resulting distribution $\rho_1^X$ satisfies
\begin{align}\label{Eq:RenyiOneStep}
       R_{q,\pi^X}(\rho_1^X) \le \frac{R_{q,\pi^X}(\rho_0^X)}{{(1 + \alpha \eta)}^{2/q}}\,.
\end{align}
Iterating this lemma for $k$ iterations yields the desired convergence rate in the theorem.
The result for KL divergence is the special case $q = 1$.

We follow the strategy and notation of Section~\ref{scn:general_strategy}.

\begin{enumerate}
    \item {\bf Forward step:}
    By Lemma~\ref{Lem:SimFlow}, along the simultaneous heat flow, 
    \begin{align*}
       \partial_t R_{q, \pi^X Q_t}(\rho_0^X Q_t) = -\frac{1}{2}\, J_{q,\pi^X Q_t}(\rho_0^X Q_t)
      \le -\frac{\alpha_t}{q}\, R_{q,\pi^X Q_t}(\rho_0^X Q_t)
   \end{align*}
   where by~\eqref{Eq:RenyiLSI}, the last inequality holds if $\pi^X Q_t$ is $\alpha_t$-LSI.
   Since $\pi^X$ satisfies $\alpha$-LSI by assumption, recall that $\pi^X Q_t = \pi^X \ast \N(0, tI)$ satisfies $\alpha_t$-LSI with $\alpha_t = (\frac{1}{\alpha} + t)^{-1} = \frac{\alpha}{1 + \alpha t}$.
   Integrating, we get 
   \begin{align*}
       R_{q,\pi^X Q_t}(\rho_0^X Q_t) \le \exp(-A_t)\, R_{q,\pi^X}(\rho_0^X)
   \end{align*}
   where 
       $A_t = \frac{1}{q} \int_0^t \alpha_s \, ds = \frac{1}{q} \int_0^t \frac{\alpha}{1 + \alpha s} \, ds
       = \frac{1}{q} \log(1+\alpha t)$.
   Therefore, after the forward step,
   \begin{align*}
       R_{q,\pi^Y}(\rho_0^Y) = R_{q,\pi^X Q_\eta}(\rho_0^X Q_\eta) \le  \frac{R_{q,\pi^X}(\rho_0^X)}{{(1 + \alpha \eta)}^{1/q}}\,.
   \end{align*}  
   
    \item {\bf Backward step:}
    By Lemma~\ref{lem:simflow2_new}, along the simultaneous backwards heat flow, 
    \begin{align*}
       \partial_t R_{q, \pi^Y Q_t^-}(\rho_0^Y Q_t^-) = -\frac{1}{2}\, J_{q,\pi^Y Q_t^-}(\rho_0^Y Q_t^-)
      \le -\frac{\alpha_{\eta-t}}{q}\, R_{q,\pi^Y Q_t^-}(\rho_0^Y Q_t^-)
   \end{align*}
   where the last inequality holds since $\pi^Y Q_t^- = \pi \ast \N(0, (\eta - t) I)$ is $\alpha_{\eta-t}$-LSI.
   Therefore, just as in the forward step, integration yields
   \begin{align*}
       R_{q,\pi^X}(\rho_1^X) = R_{q,\pi^Y Q_\eta^-}(\rho_0^Y Q_\eta^-) \le  \frac{R_{q,\pi^Y}(\rho_0^Y)}{{(1 + \alpha \eta)}^{1/q}}\,.
   \end{align*}  
\end{enumerate}
Combining the two steps above yields the desired contraction rate in~\eqref{Eq:RenyiOneStep}.

\end{proof}

\subsection{Convergence under PI}
\label{Sec:Poincare}

The dissipation of the chi-squared divergence along the Langevin dynamics is
\begin{align*}
    J_{\chi^2,\pi}(\rho)
    &:= 2 \E_\pi\Bigl[\Bigl\lVert \nabla \frac{\rho}{\pi}\Bigr\rVert^2\Bigr]\,.
\end{align*}

\begin{proof}[Proof of Theorem~\ref{Thm:Poincare}]
    We follow the strategy and notation of Section~\ref{scn:general_strategy}.
    \begin{enumerate}
        \item \textbf{Forward step:} Along the simultaneous heat flow, Lemma~\ref{Lem:SimFlow} yields
        \begin{align*}
            \partial_t \chi^2_{\pi^X Q_t}(\rho_0^X Q_t)
            &= - \frac{1}{2} \, J_{\chi^2,\pi^X Q_t}(\rho_0^X Q_t)\,, \qquad \\[0.25em]
            \partial_t R_{q,\pi^X Q_t}(\rho_0^X Q_t)
            &= - \frac{1}{2} \, J_{q,\pi^X Q_t}(\rho_0^X Q_t)\,.
        \end{align*}
        Since $\pi^X$ satisfies $\alpha$-PI, then $\pi^X Q_t$ satisfies $\alpha_t$-PI with $\alpha_t = \frac{\alpha}{1+\alpha t}$. Applying this yields
        \begin{align*}
            \partial_t \chi^2_{\pi^X Q_t}(\rho_0^X Q_t)
            &= - \frac{1}{2} \, J_{\chi^2,\pi^X Q_t}(\rho_0^X Q_t)
            \le -\alpha_t \,\chi^2_{\pi^X Q_t}(\rho_0^X Q_t)
        \end{align*}
        and therefore
        \begin{align*}
            \chi^2_{\pi^Y}(\rho_0^Y)
            &= \chi^2_{\pi^X Q_\eta}(\rho_0^X Q_\eta)
            \le \frac{\chi^2_{\pi^X}(\rho_0^X)}{1+\alpha\eta}
        \end{align*}
        upon integration.
        
        Next, from~\citet[Lemma 17]{VW19}, $\alpha_t$-PI implies
        \begin{align*}
            \partial_t R_{q,\pi^X Q_t}(\rho_0^X Q_t)
            = - \frac{1}{2} \, J_{q,\pi^X Q_t}(\rho_0^X Q_t)
            \le - \frac{2\alpha_t}{q} \, \{1-\exp(-R_{q,\pi^X Q_t}(\rho_0^X Q_t))\}\,.
        \end{align*}
        We split into two cases. If $R_{q,\pi^X}(\rho_0^X) \ge 1$, then as long as $R_{q,\pi^X Q_t}(\rho_0^X Q_t)\ge 1$ we can use the inequality $1-\exp(-x) \ge\frac{1}{2}$ for $x \ge 1$, so that
        \begin{align*}
            \partial_t R_{q,\pi^X Q_t}(\rho_0^X Q_t)
            &\le - \frac{\alpha_t}{q}\,.
        \end{align*}
        Integrating, we obtain
        \begin{align*}
            R_{q,\pi^Y}(\rho_0^Y)
            &= R_{q,\pi^X Q_\eta}(\rho_0^X Q_\eta)
            \le \Bigl(R_{q,\pi^X}(\rho_0^X) - \frac{\log(1+\alpha\eta)}{q}\Bigr) \vee 1\,.
        \end{align*}
        In the second case, if $R_{q,\pi^X}(\rho_0^X) \le 1$, then we use $1-\exp(-x) \ge \frac{x}{2}$ for $x \in [0,1]$ to obtain
        \begin{align*}
            \partial_t R_{q,\pi^X Q_t}(\rho_0^X Q_t)
            &\le -\frac{\alpha_t}{q} \, R_{q,\pi^X Q_t}(\rho_0^X Q_t)\,.
        \end{align*}
        Integrating,
        \begin{align*}
            R_{q,\pi^Y}(\rho_0^Y)
            &= R_{q,\pi^X Q_\eta}(\rho_0^X Q_\eta)
            \le \frac{R_{q,\pi^X}(\rho_0^X)}{{(1+\alpha \eta)}^{1/q}}\,.
        \end{align*}
        
        \item \textbf{Backward step:} 
        Along the simultaneous backwards heat equation, Lemma~\ref{lem:simflow2_new} yields
        \begin{align*}
            \partial_t \chi^2_{\pi^Y Q_t^-}(\rho_0^Y Q_t^-)
            &= -\frac{1}{2} \, J_{\chi^2,\pi^Y Q_t^-}(\rho_0^Y Q_t^-)\,, \qquad \\[0.25em]
            \partial_t R_{q,\pi^Y Q_t^-}(\rho_0^Y Q_t^-)
            &= -\frac{1}{2} \, J_{q,\pi^Y Q_t^-}(\rho^Y_0 Q_t^-)\,.
        \end{align*}
        Using entirely analogous arguments as in the forward step, we obtain
        \begin{align*}
            \chi^2_{\pi^X}(\rho_1^X)
            &= \chi^2_{\pi^Y Q_\eta^-}(\rho_0^Y Q_\eta^-)
            \le \frac{\chi^2_{\pi^Y}(\rho_0^Y)}{1+\alpha\eta}
        \end{align*}
        for the chi-squared divergence,
        \begin{align*}
            R_{q,\pi^X}(\rho_1^X)
            &= R_{q,\pi^Y Q_\eta^-}(\rho_0^Y Q_\eta^-)
            \le \Bigl( R_{q,\pi^Y}(\rho_0^Y) - \frac{\log(1+\alpha\eta)}{q}\Bigr) \vee 1
        \end{align*}
        for the R\'enyi divergence if $R_{q,\pi^Y}(\rho_0^Y) \ge 1$, and
        \begin{align*}
            R_{q,\pi^X}(\rho_1^X)
            &= R_{q,\pi^Y Q_\eta^-}(\rho_0^Y Q_\eta^-)
            \le \frac{R_{q,\pi^Y}(\rho_0^Y)}{{(1+\alpha\eta)}^{1/q}}
        \end{align*}
        if $R_{q,\pi^Y}(\rho_0^Y) \le 1$.
    \end{enumerate}
\end{proof}

\subsection{Convergence under LOI}\label{scn:loi_pf}

Before giving the convergence proof under LOI, we recall the following property of the behavior of LOI under convolution.

\begin{lemma}\label{lem:loi_under_convolution}
    Suppose that $\mu_0$ satisfies $(r, \alpha_0)$-LOI and $\mu_1$ satisfies $(r,\alpha_1)$-LOI.
    Then, $\mu_0\ast\mu_1$ satisfies $(r, (\alpha_0^{-1} + \alpha_1^{-1})^{-1})$-LOI.
\end{lemma}
\begin{proof}
    Let $X_0 \sim \mu_0$ and $X_1\sim\mu_1$ be independent.
    Then, we can write
    \begin{align*}
        \var_{p,\mu_0*\mu_1}(\psi)
        &= \E[\Phi(\psi^p(X_0+X_1))] - \Phi(\E[\psi^p(X_0 + X_1)])
    \end{align*}
    where $\Phi(x) := x^{2/p}$. One can then deduce the conclusion of the lemma easily from the subadditivity of the $\Phi$-entropy~\citep[Theorem 14.1]{boucheronetal2013concentration}.
\end{proof}

\begin{proof}[Proof of Theorem~\ref{thm:loi}]
    We follow the strategy and notation of Section~\ref{scn:general_strategy}.
    \begin{enumerate}
        \item \textbf{Forward step:} Along the simultaneous heat flow, Lemma~\ref{Lem:SimFlow} yields
        \begin{align*}
            \partial_t R_{q,\pi^X Q_t}(\rho_0^X Q_t)
            &= - \frac{1}{2} \, J_{q,\pi^X Q_t}(\rho_0^X Q_t)\,.
        \end{align*}
        Since $\pi^X$ satisfies $(r,\alpha)$-LOI and $\mc N(0, tI)$ satisfies $(r', t^{-1})$-LOI for any $r'\in [1,2]$~\citep[see][Corollary 1]{latalaoleszkiewicz2000loineq}, then by Lemma~\ref{lem:loi_under_convolution}, $\pi^X Q_t$ satisfies $(r,\alpha_t)$-LOI with $\alpha_t = \frac{\alpha}{1+\alpha t}$.
        
        Next, from~\citet[Theorem 2]{chewietal2021lmcpoincare}, $(r,\alpha_t)$-LOI implies
        \begin{align*}
            \partial_t R_{q,\pi^X Q_t}(\rho_0^X Q_t)
            &= - \frac{1}{2} \, J_{q,\pi^X Q_t}(\rho_0^X Q_t) \\
            &\le - \frac{\alpha_t}{136q} \begin{cases} {R_{q,\pi^X Q_t}(\rho_0^X Q_t)}^{2-2/r}\,, & R_{q,\pi^X Q_t}(\rho_0^X Q_t) \ge 1\,, \\
            R_{q,\pi^X Q_t}(\rho_0^X Q_t)\,, & R_{q,\pi^X Q_t}(\rho_0^X Q_t) \le 1\,. \end{cases}
        \end{align*}
        We split into two cases. If $R_{q,\pi^X}(\rho_0^X) \ge 1$, then as long as $R_{q,\pi^X Q_t}(\rho_0^X Q_t)\ge 1$,
        \begin{align*}
            \partial_t\, {R_{q,\pi^X Q_t}(\rho_0^X Q_t)}^{2/r-1}
            &= \bigl( \frac{2}{r} -1\bigr) \, \frac{\partial_t R_{q,\pi^X Q_t}(\rho_0^X Q_t)}{{R_{q,\pi^X Q_t}(\rho_0^X Q_t)}^{2-2/r}}
            \le - \frac{\alpha_t}{136q} \, \bigl( \frac{2}{r} -1\bigr)
        \end{align*}
        and therefore
        \begin{align*}
            {R_{q,\pi^Y}(\rho_0^Y)}^{2/r-1}
            &= {R_{q,\pi^X Q_\eta}(\rho_0^X Q_\eta)}^{2/r-1} \\
            &\le \Bigl({R_{q,\pi^X}(\rho_0^X)}^{2/r-1} - \frac{(2/r - 1) \log(1+\alpha\eta)}{136q}\Bigr) \vee 1\,.
        \end{align*}
        In the second case, if $R_{q,\pi^X}(\rho_0^X) \le 1$, then
        \begin{align*}
            \partial_t R_{q,\pi^X Q_t}(\rho_0^X Q_t)
            &\le -\frac{\alpha_t}{136q} \, R_{q,\pi^X Q_t}(\rho_0^X Q_t)\,.
        \end{align*}
        Integrating,
        \begin{align*}
            R_{q,\pi^Y}(\rho_0^Y)
            &= R_{q,\pi^X Q_\eta}(\rho_0^X Q_\eta)
            \le \frac{R_{q,\pi^X}(\rho_0^X)}{{(1+\alpha \eta)}^{1/(136q)}}\,.
        \end{align*}
        \item \textbf{Backward step:} 
        Along the simultaneous backwards heat equation, Lemma~\ref{lem:simflow2_new} yields
        \begin{align*}
            \partial_t R_{q,\pi^Y Q_t^-}(\rho_0^Y Q_t^-)
            &= -\frac{1}{2} \, J_{q,\pi^Y Q_t^-}(\rho^Y_0 Q_t^-)\,.
        \end{align*}
        Using entirely analogous arguments as in the forward step, we obtain
        \begin{align*}
            {R_{q,\pi^X}(\rho_1^X)}^{2/r-1}
            &= {R_{q,\pi^Y Q_\eta^-}(\rho_0^Y Q_\eta^-)}^{2/r-1} \\
            &\le \Bigl( {R_{q,\pi^Y}(\rho_0^Y)}^{2/r-1} - \frac{(2/r-1)\log(1+\alpha\eta)}{136q}\Bigr) \vee 1
        \end{align*}
        if $R_{q,\pi^Y}(\rho_0^Y) \ge 1$, and
        \begin{align*}
            R_{q,\pi^X}(\rho_1^X)
            &= R_{q,\pi^Y Q_\eta^-}(\rho_0^Y Q_\eta^-)
            \le \frac{R_{q,\pi^Y}(\rho_0^Y)}{{(1+\alpha\eta)}^{1/(136q)}}
        \end{align*}
        if $R_{q,\pi^Y}(\rho_0^Y) \le 1$.
    \end{enumerate}
\end{proof}

\subsection{The proximal sampler as an entropy-regularized Wasserstein gradient flow}
\label{Sec:EntReg}

\begin{proof}[Proof of Theorem \ref{thm:JKO}]
Plugging \eqref{eq:entropyW} into \eqref{eq:forwardJKO} yields $\rho_k^Y = \gamma^Y$ with $\gamma$ being the solution to
	\[
		\min_{\substack{\gamma\in\mc P_2(\R^d\times\R^d) \\ \gamma^X = \rho_k^X}}{\Bigl\{ \int \frac{1}{2\eta} \,\|x-y\|^2 \,d\gamma(x,y) + H(\gamma)\Bigr\}}\,,
	\]
which is clearly $\gamma(x,y) \propto \rho_k^X(x)\exp(-\frac{1}{2\eta}\,\|x-y\|^2)$. Thus, $\rho_k^Y =\gamma^Y=\rho_k^X*\N(0, \eta I)$.

Similarly, plugging \eqref{eq:entropyW} into \eqref{eq:backwardJKO} yields $\rho_{k+1}^X = \gamma^X$ with $\gamma$ being the solution to 
	\[
		\min_{\substack{\gamma \in\mc P_2(\R^d\times \R^d) \\ \gamma^Y = \rho_k^Y}}{\Bigl\{\int \bigl[f(x)+\frac{1}{2\eta} \,\|x-y\|^2\bigr]\,d\gamma(x,y) + H(\gamma)\Bigr\}}\,,
	\]
which is clearly $\gamma(x,y) \propto \rho_k^Y(y) \exp(-f(x)-\frac{1}{2\eta}\,\|x-y\|^2)$. Thus, $\rho_{k+1}^X$ is induced by $\pi^{X|Y}(x \mid y) \propto_x \exp(-f(x)-\frac{1}{2\eta}\,\|x-y\|^2)$ from marginal distribution $Y\sim \rho_k^Y$. 
\end{proof}

\subsection{The proximal point method as a limit of the proximal sampler}
\label{Sec:ProxLimit}

\begin{proof}[Proof of Theorem \ref{Thm:Limit}]
With a general $\epsilon$, following similar argument as in Section~\ref{Sec:EntReg}, we can show that the updates \eqref{eq:forwardJKOgeneral}--\eqref{eq:backwardJKOgeneral} correspond to the sampling algorithm
	\begin{subequations}\label{eq:RGO_eps}
	\begin{eqnarray}\label{eq:RGO_eps1}
		&y_{k} &\sim \pi_\epsilon^{Y|X=x_k} = \N(x_k, \epsilon \eta I)\,,		\\
		&x_{k+1} &\sim \pi_\epsilon^{X|Y=y_{k}} \propto \exp\Bigl[-\frac{1}{\epsilon} \,\bigl(f(x)+\frac{1}{2\eta}\,\|x-y_{k}\|^2\bigr)\Bigr]\,. \label{eq:RGO_eps2}
	\end{eqnarray}
	\end{subequations}
As $\epsilon\searrow 0$, we see that \eqref{eq:RGO_eps1} converges to  $y_{k} = x_k$, whereas \eqref{eq:RGO_eps2} converges to the proximal mapping $x_{k+1} = \argmin_{x\in\R^d}\{ f(x)+\frac{1}{2\eta}\,\|x-y_k\|^2\}$. Combining the two gives exactly the proximal point update $x_{k+1} = \prox_{\eta f}(x_k)$. In addition, the invariant distribution of this algorithm is $\pi^X_\epsilon \propto \exp(-f/\epsilon)$, which converges to a Dirac distribution concentrating on the minimizer of $f$ (or a uniform distribution over the minimizer set of $f$).
\end{proof}

It turns out that under some assumptions, the convergence rate of the updates \eqref{eq:RGO_eps} is independent of the entropy regularization level $\epsilon$. 
We state and prove the result below for KL divergence only, but the result also holds for R\'enyi divergence and $\chi^2$-divergence.

\begin{theorem}\label{thm:epsilon}
When $f$ is $\alpha$-strongly convex, the updates of the generalized proximal sampler algorithm converge to the stationary distribution $\pi^X_\epsilon \propto \exp(-f/\epsilon)$ with rate
	\begin{equation}
		H_{\pi^X_\epsilon} (\rho^X_k) \le \frac{1}{{(1+\alpha \eta)}^{2k}}\, H_{\pi^X_\epsilon} (\rho^X_0)\,. 
	\end{equation}
\end{theorem}
\begin{proof}
The forward step \eqref{eq:RGO_eps1} can be modeled by the scaled diffusion 
    \begin{equation}\label{eq:scaled_heat_flow}
        \partial_t \rho_t = \frac{\epsilon}{2}\,\Delta \rho_t
    \end{equation}
over the time interval $[0, \eta]$. Let ${(Q_t^\epsilon)}_{t\ge 0}$ denote the heat semigroup corresponding to~\eqref{eq:scaled_heat_flow}. It follows from Lemma~\ref{Lem:SimFlow} that
    \begin{equation}\label{eq:dHdt}
        \partial_t H_{\pi^X Q_t^\epsilon}(\rho_0^X Q_t^\epsilon) = -\frac{\epsilon}{2} \, J_{\pi^X Q_t^\epsilon}(\rho_0^X Q_t^\epsilon)\,.
    \end{equation}
Apparently, $\pi^X Q_t^\epsilon = \pi^X_\epsilon * \N(0, \epsilon t I)$. Thus, $\pi^X Q_t^\epsilon$ satisfies $\alpha_t$-LSI with 
    \begin{equation}\label{eq:alphat}
        \alpha_t = \frac{1}{\frac{\epsilon}{\alpha}+\epsilon t} = \frac{\alpha}{\epsilon\,(1+\alpha t)}\,,
    \end{equation}
where in the above we have used the fact that $\exp(-f/\epsilon)$ satisfies $\frac{\alpha}{\epsilon}$-LSI when $f$ is $\alpha$-strongly convex.
Plugging \eqref{eq:alphat} into \eqref{eq:dHdt} yields 
    \begin{equation}
        \partial_t H_{\pi^X Q_t^\epsilon}(\rho_0^X Q_t^\epsilon)
        \le -\alpha_t \, H_{\pi^X Q_t^\epsilon}(\rho_0^X Q_t^\epsilon)\,.
    \end{equation}
Thus, as before,
    \begin{equation}
        H_{\pi^Y_\epsilon}(\rho_0^Y) = H_{\pi^X Q_\eta^\epsilon}(\rho_0^X Q_\eta^\epsilon)\le \frac{1}{1+\alpha \eta} \,H_{\pi^X}(\rho_0^X)\,.
    \end{equation}
The contraction rate in the backward direction is the same and the proof is similar to that of Theorem~\ref{Thm:LSI}. This completes the proof.
\end{proof}

Theorem \ref{thm:epsilon} is true as long as $\exp(-f/\epsilon)$ satisfies $(\alpha/\epsilon)$-LSI. The latter is ensured when $f$ is $\alpha$-strongly convex; we ask whether it remains true under a weaker condition on $f$ (such as $\alpha$-PL).

\section{Optimization proofs inspired by the proximal sampler}

\subsection{Alternative proof of the contractivity of the proximal map}\label{scn:contractivity_proximal_map}

The following theorem is well-known in optimization.

\begin{theorem}\label{Thm:ProxCont}
Let $f : \R^d\to\R$ be $\alpha$-strongly convex and differentiable. Then, the proximal mapping
\begin{align*}
    \prox_{\eta f}(y)
    &:= \argmin_{x\in\R^d}{\Bigl\{ f(x) + \frac{1}{2\eta} \, \norm{x-y}^2\Bigr\}}
\end{align*}
is a $\frac{1}{1+\alpha\eta}$-contraction.
\end{theorem}

Here, we give a new proof of the theorem which translates the convergence proof of the proximal sampler in~\citet{LST21arxiv} to optimization. 

We recall that $\alpha$-strong convexity implies the $\alpha$-PL inequality (or gradient domination inequality)
\begin{align*}
    \norm{\nabla f(x)}^2
    &\ge 2\alpha \, \{f(x) - \min f\} \qquad\text{for all}~x\in\R^d\,,
\end{align*}
which in turn implies the $\alpha$-quadratic growth inequality
\begin{align*}
    f(x) - \min f
    &\ge \frac{\alpha}{2} \, \norm{x-x^\star}^2 \qquad\text{for all}~x\in\R^d\,,
\end{align*}
with $x^\star = \argmin f$, see~\citet{ottovillani2000lsi,blanchet2018family}.
\medskip{}

\begin{proof}[Proof of Theorem~\ref{Thm:ProxCont}]
    Let $f_x(z) := f(z) + \frac{1}{2\eta} \, \norm{x-z}^2$, and define $f_y$ similarly.
    Then,
    \begin{align*}
        x'
        &:= \prox_{\eta f}(x)
        = \argmin f_x\,, \\
        y'
        &:= \prox_{\eta f}(y)
        = \argmin f_y\,.
    \end{align*}
    Since $f_x$ is $(\alpha + \frac{1}{\eta})$-strongly convex, then by applying the quadratic growth and PL inequalities,
    \begin{align*}
        \norm{x'-y'}^2
        &\le \frac{2}{\alpha + 1/\eta} \, \{f_x(y') - f_x(x')\}
        \le \frac{1}{{(\alpha + 1/\eta)}^2} \, \norm{\nabla f_x(y')}^2 \\
        &= \frac{1}{{(\alpha + 1/\eta)}^2} \, \bigl\lVert \nabla f(y') + \frac{1}{\eta} \, (y' - x) \bigr\rVert^2 \\
        &= \frac{1}{{(\alpha + 1/\eta)}^2} \, \bigl\lVert - \frac{1}{\eta} \, (y' - y) + \frac{1}{\eta} \, (y' - x) \bigr\rVert^2
        = \frac{1}{{(1 +  \alpha\eta)}^2} \, \norm{x-y}^2 
    \end{align*}
    where the last line uses the optimality condition $\nabla f(y') + \frac{1}{\eta} \, (y' - y) = 0$ from the definition of $y'$.
\end{proof}

By comparing with the proof of~\citet[Lemma 2]{LST21arxiv}, we see that $f_y$
is analogous to $H_{\pi^{X \mid Y=y}}$ for the proximal sampler.

At first glance, it may appear that the proof above only requires a PL inequality, and not strong convexity.
However, this is not the case, as it in fact requires that $f_x$ satisfies $(\alpha+1/\eta)$-PL, which does not follow from (for example) the assumption that $f$ satisfies $\alpha$-PL. 

\subsection{Optimal contraction factor for the proximal point method under PL}
\label{Sec:ProofProxLSI}

Our proof uses the Hopf--Lax semigroup, guided by the following intuition.
There is an analogy between the standard algebra $(+,\times)$ and the tropical algebra $(\inf, +)$; see e.g.~\citet[Section 9.4]{baccellietal1992synchronization} or~\citet[Lecture 16]{ambrosio2021lectures}.
The following table describes these analogies.
    \begin{align*}
        \begin{array}{cc}
            (+,\times) & (\inf, +) \\
            \text{convolution} & \text{inf-convolution} \\
            \text{Fourier transform} & \text{convex conjugate} \\
            \text{diffusion} & \text{gradient flow} \\
            \text{heat equation} & \text{Hamilton-Jacobi equation} \\
            \text{heat semigroup} & \text{Hopf-Lax semigroup}
        \end{array}
    \end{align*}
As described in Section~\ref{scn:general_strategy}, our proofs for the proximal sampler involve computing the time derivative of $t\mapsto H_{\pi^X Q_t}(\rho_0^X Q_t)$ where ${(\pi^X Q_t)}_{t\ge 0}$, ${(\rho_0^X Q_t)}_{t\ge 0}$ are simultaneously evolving according to the heat flow.
In what follows, we will consider the time derivative of $t\mapsto f_{t}(x)$, where $f_t$ is the Moreau envelope of $f$. 

\medskip{}

\begin{proof}[Proof of Theorem~\ref{Thm:ProxLSI}]
    Let us define, for $t > 0$,
    \begin{align}\label{eq:simul_hj_def}
        f_{t,x}(z)
        &:= f(z) + \frac{1}{2t} \, \norm{z - x}^2\,, \qquad x_t := \argmin f_{t,x}\,.
    \end{align}
    Then $x_t = \prox_{t f}(x)$ and $x \mapsto f_{t,x}(x_t)$ is the Moreau envelope of $f$. 
    Recall the optimality condition
    \begin{align*}
        \nabla f(x_t) + \frac{1}{t} \, (x_t - x) = 0\,.
    \end{align*}
    The Moreau envelope satisfies the Hamilton Jacobi equation
    $$\partial_t f_{t,x}(x_t)
       = \langle \underbrace{\nabla f_{t,x}(x_t)}_{=0}, \dot x_t \rangle - \frac{1}{2t^2} \, \norm{x_t - x}^2.$$
    Using PL inequality,
    \begin{align*}
        \partial_t f_{t,x}(x_t)
        &= - \frac{\alpha}{2t \, (1+\alpha t)} \, \norm{x_t - x}^2 - \frac{1}{2t^2 \, (1+\alpha t)} \, \norm{x_t - x}^2 \\
        &= - \frac{\alpha}{2t \, (1+\alpha t)} \, \norm{x_t - x}^2 - \frac{1}{2\, (1+\alpha t)} \, \norm{\nabla f(x_t)}^2 \\
        &\le - \frac{\alpha}{2t \, (1+\alpha t)} \, \norm{x_t - x}^2 - \frac{\alpha}{1+\alpha t} \, \{f(x_t) - f^\star\}
    \end{align*}
    which yields
    \begin{align*}
        \partial_t \{f_{t,x}(x_t) - f^\star\}
        &\le - \frac{\alpha}{1+\alpha t} \, \{f_{t,x}(x_t) - f^\star\}\,.
    \end{align*}
    Integrating this yields\footnote{Denote by ${(Q_t^{\rm HL})}_{t\ge 0}$ the Hopf--Lax semigroup defined by $Q_t^{\rm HL} f(x) = f_{t,x}(x_t)$. One can check that $Q_t^{\rm HL} f(x^\star) = f(x^\star)$ where $x^\star = \argmin f$. So, we can rewrite this inequality as $Q_t^{\rm HL} f(x) - Q_t^{\rm HL} f(x^\star) \leq \frac{1}{(1+\alpha t)} \,\{f(x) - f(x^\star)\}$.}
    \begin{align*}
        f_{\eta,x}(x_\eta) - f^\star
        &\le \{f(x) - f^\star\} \exp\Bigl(-\int_0^\eta \frac{\alpha}{1+\alpha t} \,\D t \Bigr)
        = \frac{1}{1+\alpha \eta} \, \{f(x) - f^\star\}\,.
    \end{align*}
    
    Hence,
    \begin{align*}
        \frac{1}{1+\alpha \eta} \, \{f(x) - f^\star\}
        &\ge f(x') - f^\star + \frac{1}{2\eta} \, \norm{x' - x}^2
        = f(x') - f^\star + \frac{\eta}{2} \, \norm{\nabla f(x')}^2 \\
        &\ge f(x') - f^\star + \alpha \eta \, \{f(x') - f^\star\}
        = (1+\alpha \eta) \, \{f(x') - f^\star\}\,.
    \end{align*}
    This completes the proof.
\end{proof}

\bibliographystyle{plainnat}
\bibliography{ref.bib}

\end{document}